\documentclass[12pt,a4paper]{amsart}
\calclayout
\numberwithin{equation}{section}
\allowdisplaybreaks

\newcounter{constant}
\newcommand{\C}{\refstepcounter{constant}\ensuremath{C_{\theconstant}}}
\newcommand{\Cr}[1]{\ensuremath{C_{\ref{#1}}}}

\theoremstyle{plain}
\newtheorem{theorem}{Theorem}[section]
\newtheorem{lemma}[theorem]{Lemma}
\newtheorem{proposition}[theorem]{Proposition}

\theoremstyle{remark}
\newtheorem{remark}[theorem]{Remark}
\newtheorem*{remark*}{Remark}

\usepackage{amssymb}
\usepackage{mathtools}\mathtoolsset{showonlyrefs}
\usepackage{mathrsfs}
\usepackage{comment}
\usepackage{hyperref}
\usepackage[msc-links]{amsrefs}
\usepackage{xcolor}
\usepackage{environ}
\usepackage{enumitem}

\newcounter{statement}
\newenvironment{statement}{\setcounter{statement}{\value{equation}} \list{(\theequation)}{\usecounter{equation}} \setcounter{equation}{\value{statement}} \item }{\endlist}

\newcommand{\Diff}{\mathrm{d}} 
\newcommand{\FntMsrSp}{\mathcal{M}} 
\newcommand{\RealLine}{\mathbb{R}} 
\newcommand{\NullMsr}{\mathbf{0}} 
\newcommand{\bfQ}{\mathbf{Q}} 
\newcommand{\bfP}{\mathbf{P}} 
\newcommand{\Borel}{\mathscr{B}} 
\newcommand{\srF}{\mathscr{F}} 
\newcommand{\Bndd}{\mathrm{b}} 
\newcommand{\Pos}{\mathrm{p}} 
\newcommand{\BnddPos}{\mathrm{bp}} 
\newcommand{\ProbSupProc}{\mathrm{P}} 
\newcommand{\ProbQProc}{\widetilde{\mathrm{P}}} 
\newcommand{\ProbSpnDec}{\mathrm{Q}}
\newcommand{\Cadlag}{c\`adl\`ag } 
\newcommand{\Spn}{{\tilde\xi}} 
\newcommand{\Ind}{\mathbf{1}} 
\newcommand{\NMsr}{\mathbb N} 

\begin{document}

\title[Subcritical superprocesses]
{Subcritical superprocesses conditioned on non-extinction}

\thanks{The research of this project is supported by the National Key R\&D Program of China (No. 2020YFA0712900).}

\author[R. Liu]{Rongli Liu}
\address[R. Liu]{Mathematics and Applied Mathematics\\ Beijing jiaotong University\\ Beijing 100044\\ P. R. China}
\email{rlliu@bjtu.edu.cn}
\thanks{The research of Rongli Liu is supported in part by NSFC (Grant No. 12271374), and the Fundamental Research Funds for the Central Universities (Grant No.  2017RC007)}

\author[Y.-X. Ren]{Yan-Xia Ren}
\address[Y.-X. Ren]{LMAM School of Mathematical Sciences \& Center for
Statistical Science\\ Peking University\\ Beijing 100871\\ P. R. China}
\email{yxren@math.pku.edu.cn}
\thanks{The research of Yan-Xia Ren is supported in part by NSFC (Grant Nos. 12071011 and 11731009)  and LMEQF}

\author[R. Song]{Renming Song}
\address[R. Song]{Department of Mathematics\\ University of Illinois at Urbana-Champaign \\ Urbana \\ IL 61801\\ USA}
\email{rsong@illinois.edu}
\thanks{The research of Renming Song is supported in part by a  grant from the Simons Foundation (\#960480, Renming Song)}

\author[Z. Sun]{Zhenyao Sun}
\address[Z. Sun]{The Faculty of Industrial Engineering and Management \\ Technion --- Israel Institute of Technology \\ Haifa 3200003\\ Israel}
\email{zhenyao.sun@gmail.com}

\begin{abstract}
We consider a class of subcritical superprocesses $(X_t)_{t\geq 0}$ with general spatial motions and general branching mechanisms. We study the asymptotic behaviors of $\mathbf Q_{t,r}$, the distribution of $X_t$ conditioned on $X_{t+r}$ not being a null measure. We first give the existence of $\lim_{t\to \infty}\mathbf Q_{t,r}$ and $\lim_{r\to \infty}\mathbf Q_{t,r}$, and then show that an $L\log L$-type condition is equivalent to the existence of the double limits: $\lim_{r\to \infty} \lim_{t\to\infty}\mathbf Q_{t, r}$ and $\lim_{t\to \infty} \lim_{r\to\infty}\mathbf Q_{t, r}$. Finally, when the $L\log L$-type condition holds, we show that those double limits, and $\lim_{r,t\to \infty}\mathbf Q_{t,r}$, are the same. 
\end{abstract}

\maketitle
\section{Introduction} \label{sec:I}
\subsection*{Motivation}
The study of the extinction of stochastic processes related to population dynamics is of great interest in both biology and probability theory.
Take a subcritical Galton-Watson process $(Z_n)_{n\in \mathbb Z_+}$ as an example.
Assume that $Z_0 = 1$ and $m = E[Z_1] \in (0,1)$. It is well known that the extinction probability $q:=\lim_{n\to \infty}P(Z_n = 0)$
is equal to 1.
In other words, the probability $P(Z_n>0)$ decays to $0$.
A natural question is  to find the decay rate of this probability.
In 1967, Heathcote, Seneta and Vere-Jones \cite{MR0217889} proved that the following three statements are equivalent.
\begin{statement}
$\lim_{n\rightarrow\infty}P(Z_n>0)/m^n>0$.
\end{statement}
\begin{statement}
$\sup E[Z_n|Z_n>0]<\infty$.
\end{statement}
\begin{statement}
\label{eq:I.01} $E\left[Z_1\log^+ Z_1\right]<\infty$.
\end{statement}
Condition \eqref{eq:I.01} is now known as the $L\log L$ condition and the equivalence of the three statements above is usually called the $L\log L$ criterion.
It is also natural to consider $Q_{n,0}$, the distribution of $Z_n$ conditioned on $\{Z_n > 0\}$.
In 1967,  Heathcote, Seneta and Vere-Jones \cite{MR0217889} and Joffe \cite{MR205337} proved that $Q_{n,0}$ has a weak limit $Q_{\infty,0}$ when $n\to\infty$.
This result was first obtained by Yaglom \cite{MR0022045} in 1947 under some moment condition, and the probability measure $Q_{\infty,0}$ is therefore referred to as the Yaglom limit.
One can also consider $Q_{n,m}$, the distribution of $Z_n$ conditioned on $\{Z_{n+m}>0\}$.
As a corollary of the Yaglom limit result, Athreya and Ney \cite{MR0373040} showed in 1972 that for every $m\in \mathbb Z_+$, $Q_{n,m}$ has a weak limit $Q_{\infty,m}$ when $n\to \infty$.
Joffe, in his 1967 paper \cite{MR205337},
 pointed out that
for every $n\in \mathbb Z_+$, $Q_{n,m}$ has a weak limit $Q_{n,\infty}$ when $m\to \infty$.
Later in 1999, Pakes \cite{MR1735780} proved that the $L\log L$
condition
\eqref{eq:I.01} is equivalent to each of the following two statements.
\begin{statement}
$Q_{\infty,m}$ has a weak limit when $m \to \infty$.
\end{statement}
\begin{statement}
$Q_{n,\infty}$ has a weak limit when $n\to \infty$.
\end{statement}
Moreover, when $\eqref{eq:I.01}$ holds, Pakes \cite{MR1735780} showed that $\lim_{m\to \infty} Q_{\infty,m} = \lim_{n\to \infty} Q_{n,\infty}$.

Yaglom limit theorem is now a fundamental topic in the study of Markov processes.
A long list of references
on Yaglom limit theorems of a variety of models
can be found on the website \cite{pollett2008quasi} maintained by Pollett.
It turns out that Heathcote, Seneta and Vere-Jones' $L\log L$ theorem, as well as Pakes' double limit theorem, are also universal among
models with the Markovian branching property.
Analogs of these results in the context of multitype Galton-Watson
processes can be found in \cite{MR3476213} and the references therein.
Results for continuous-state branching processes can be found in
\cites{MR408016,MR2299923} and \cite{MR1727226}.

We are interested in a class of measure-valued branching processes known
as superprocesses.
The book \cite{MR2760602} is a good reference for superprocesses.
In recent years, there have been a lot of papers on the large time asymptotic behavior of superprocesses.
For laws of large numbers and central limit theorems of some supercritical superprocesses,
see \cites{MR4397885,MR2397881,MR3010225,MR3395469,MR3293289,MR4049077}
and the references therein.
For Yaglom limit results of various critical superprocesses, see
 \cites{MR1088825,MR3459635,MR4058118,MR4102269}.

In our recent work \cite{MR4175472}, we characterized the Yaglom limits of a class of subcritical superprocesses with general spatial motions and general branching mechanisms.
The goal of this paper
is to establish Heathcote, Seneta and Vere-Jones' $L\log L$ theorem, as well as Pakes' double limit theorem, for the same class of subcritical superprocesses.

\subsection*{Model and Assumptions}
We first recall the definition of superprocesses.
For any metric space $F$,
we use $\Bndd\Borel (F)$
($\mathrm p\Borel(F)$,
$\BnddPos\Borel (F)$) to denote the set of bounded Borel  functions on $F$ (non-negative Borel  functions on $F$, non-negative bounded Borel functions on $F$, respectively).
Let $E$ be a Polish space.  Let $(\xi_t)_{t\in [0,\zeta)}$ be an $E$-valued Borel right process
with (possibly sub-Markovian) transition semigroup $(P_t)_{t\geq 0}$
and lifetime $\zeta$.
 Denote $\mathbb R_+:=[0,\infty)$.  Let $\psi$ be a function on $E \times \RealLine_+$ given by
\begin{align}
	\psi(x,z)
= -\beta(x) z + \sigma(x)^2 z^2 + \int_{(0,\infty)} (e^{-zu} -1 + zu) \pi(x,\Diff u),
	\quad x\in E, z\geq 0,
\end{align}
where
$\beta, \sigma\in \Bndd\Borel (E)$,
and $\pi$ is a kernel from $E$ to $(0,\infty)$ such that \[\sup_{x\in E} \int_{(0,\infty)} (u\wedge u^2) \pi(x,\Diff u) <\infty.\]
For any $f\in \mathrm b\mathrm p\Borel (E)$,  there exists a unique non-negative Borel function $(t,x)\mapsto V_tf(x)$ on $\RealLine_+\times E$ such that $\sup_{0\leq t\leq t_0, x\in E} V_tf(x) < \infty$ for every $t_0\geq 0,$ and that
	\begin{equation} 		
    V_tf(x) + \int_0^t  \Diff s \int_E  \psi\big(y, V_{t-s} f(y)\big) P_s(x, \Diff y)  = P_tf(x), \quad t\geq 0, x\in E.
	\end{equation}
	The Polish space of all finite Borel measures on $E$, equipped with the topology of weak convergence, is denoted by $\FntMsrSp$.
	It is known that there exists an $\FntMsrSp$-valued conservative right process $(X_t)_{t\geq 0}$ with transition semigroup $(Q_t)_{t\geq 0}$ such that for each $\mu \in \FntMsrSp, t\in \RealLine_+$ and $f\in \mathrm b\mathrm p\Borel (E)$,
\begin{equation} \label{eq:I.13}
	\int_{\FntMsrSp} \exp\Big\{- \int_E f(x)\eta(\Diff x)\Big\} Q_t(\mu, \Diff\eta)
	= \exp\Big\{- \int_E V_tf(x)\mu(\Diff x)\Big\}.
\end{equation}
This process $(X_t)_{t\geq 0}$ is known as a $(\xi, \psi)$-superprocess.
	We refer our readers to \cite{MR2760602} for more details.

For each $x\in E$, denote by $\Pi_x$ the law of
$(\xi_t)_{t\in [0,\zeta)}$ with initial value $\xi_0 = x$.
For each $\mu \in \mathcal M$, denote by $\ProbSupProc_\mu$ the law of
$(X_t)_{t\geq 0}$ with initial value $X_0 = \mu$.
Given any measure $\gamma$ and function $f$, we write $\gamma(f)$ for the integral of $f$ with respect to $\gamma$ whenever it is well-defined.
For any $f\in \mathrm b\Borel(E)$, define
\begin{equation}
T_tf(x)=
\Pi_x\big[e^{\int^t_0 \beta(\xi_s) \mathrm ds}f(\xi_t)\mathbf 1_{\{t<\zeta\}}\big],
\quad t\geq 0, x\in E.
\end{equation}
It is known that $(T_t)_{t\geq 0}$ is a Borel semigroup on $E$, and that
	\begin{equation} \label{eq:I.2}
	\mu (T_t f) = \ProbSupProc_\mu[X_t(f)],\quad \mu \in \FntMsrSp, t\in \RealLine_+, f\in \mathrm b\Borel(E).
\end{equation}
We call $(T_t)_{t\geq 0}$ the mean semigroup of $X$.
We will always assume the following statement holds.
	\begin{statement} \label{asp:H1}
	There exist a constant $\lambda \in \RealLine$, a bounded strictly positive Borel function $\phi$ on $E$, and a probability measure $\nu$ with full support on $E$, such that $\nu(\phi) = 1$, and that for any $t\geq 0$, $T_t \phi = e^{\lambda t} \phi$ and $\nu T_t = e^{\lambda t} \nu$.
	\end{statement}
Denote by $L^+_1(\nu)$ the collection of $f\in \mathrm p\Borel(E)$ such that $\nu(f)<\infty$.
We further assume that the mean semigroup $(T_t)_{t\geq 0}$
satisfies the following condition.
\begin{statement}	\label{asp:H2}
		There exists a map $(t,x,f)\mapsto H_tf(x)$ from $(0,\infty)\times E \times L_1^+(\nu)$ to $\RealLine$ such that
	 $T_t f(x)=e^{\lambda t}\phi(x)\nu(f)\big(1+ H_tf(x)\big)$ for any $t>0$, $x\in E$ and $f\in L_1^+(\nu)$;
	$\sup_{x\in E, f\in L_1^+(\nu)} |H_tf(x)|$ $<$ $\infty$ for any $t > 0$;  and
	$ \lim_{t\to \infty} \sup_{x\in E, f\in L_1^+(\nu)} |H_tf(x)| = 0.$
\end{statement}
	The triplet $(\lambda, \phi, \nu)$ satisfying \eqref{asp:H1} and \eqref{asp:H2} is unique.
	In fact, suppose that there is another triplet $(\lambda', \phi', \nu')$ satisfying \eqref{asp:H1} and \eqref{asp:H2}, then  $e^{-\lambda t}T_t \mathbf 1_E (x) \to \phi(x)$ and $e^{-\lambda' t} T_t \mathbf 1_E(x) \to \phi'(x)$ as $t\to \infty$ for arbitrary $x\in \mathbb E$ .
	This can only happen if $(\lambda',\phi') = (\lambda, \phi)$.
	Also it is clear that for every bounded Borel function $f$ on $E$,
\begin{equation}
	\nu'(f) = \lim_{t\to \infty} \frac{T_tf(x)}{e^{\lambda't} \phi'(x)} =  \lim_{t\to \infty} \frac{T_tf(x)}{e^{\lambda t} \phi(x)} = \nu(f),
\end{equation}
	which says that $\nu = \nu'$.

	Assumptions similar to \eqref{asp:H1} and \eqref{asp:H2} are nowadays very common in the study of superprocesses \cites{MR2535827,MR3459635,MR4058118,MR4102269,MR4175472,garcia2022asymptotic} and other spatial Markovian branching processes
\cites{MR3601657,garcia2022asymptotic,MR4187129,harris2021yaglom,MR4187122}.
	In particular, we
	mention a very recent paper \cite{garcia2022asymptotic} where exactly the same assumptions
	were used to study the asymptotic  behavior of the moments of both the superprocesses and the branching Markov processes.
	In general, it was explained in our earlier paper \cite{MR4175472} that \eqref{asp:H1} and \eqref{asp:H2} hold true if the transition semigroup of the Markov process $(\xi_t)_{t\geq 0}$ is intrinsiclly ultracontractive.
	(For the definition and more details on the intrinsically ultracontractivity, see \cite{MR2487824} and \cite{MR3601657}.)
	Some interesting examples satisfying \eqref{asp:H1} and \eqref{asp:H2} include
	multitype irreducible continuous-state branching processes and super-Brownian motions in a bounded Lipshitz domain.
	Many more examples can be found in \cite{MR4175472}*{Section 1.3} and \cite{MR3459635}*{Section 1.4}.
         We also  mention here that one can not apply our results to the super-Brownian motion on $\mathbb R^d$ because it does not satisfy \eqref{asp:H2}.

	Under the Assumption \eqref{asp:H1} and \eqref{asp:H2},	we say the superprocess is supercritical, critical, or subcritical, if $\lambda > 0$, $\lambda = 0$, or $\lambda < 0$, respectively.
Note that $(e^{-\lambda t} X_t(\phi))_{t\geq 0}$ is a martingale.
 The process grows on average if $\lambda>0$; decays on average if $\lambda<0$; maintains a stabilizing
average if $\lambda=0$. The above definition is consistent
in spirit with the same notion for Galton-Watson processes.
See \cite{harris2021yaglom} for similar definitions for branching Markov processes.
	In this paper,	we are only concerned with the subcritical case, i.e.,
	we will assume that
	\begin{equation}\label{asp:15}
		\lambda < 0.
	\end{equation}

Denote by $\NullMsr$ the null measure on $E$.
Define $\FntMsrSp^o:= \FntMsrSp \setminus \{\NullMsr\}$.
It is possible that the superprocess is persistent in the sense that  $\ProbSupProc_\mu(X_t\neq\NullMsr \text{ for all } t\geq 0) = 1$ for any $\mu\in \FntMsrSp^o$.
Note that, if $X$ is persistent, then it is trivial to consider $X$
conditioned on non-extinction.
So we use the following assumption to exclude this trivial case:
\begin{equation}
	\label{asp:H3}
\mathrm P_\nu(X_t = \NullMsr)>0,\quad t> 0.
\end{equation}
	If the branching mechanism is spatially homogeneous, that is to say the function $\psi(x,z) = \psi(z)$ is independent of $x\in E$, then \eqref{asp:H3} is known to be equivalent to Grey's condition:
\begin{statement}
	There exists $z'>0$ such that $\psi(z) > 0$ for all $z\geq z'$ and $\int_{z'}^\infty \psi(z)^{-1}\mathrm dz < \infty$.
\end{statement}
	It is also known that if the branching mechanism $\psi(x,z)$ is bounded below by a spatially homogeneous branching mechanism $\tilde\psi$ satisfying Grey's condition, then \eqref{asp:H3} holds.
	See \cite{MR3459635}*{Lemma 2.3} for more details.

\subsection*{Main results}
Given $X_0 = \mu \in \FntMsrSp^o$, we denote by
$\bfQ^\mu_{t,r}$ the distribution of $X_t$ conditioned on $\{X_{t+r}\neq \NullMsr\}$, i.e.,
\[
	\bfQ^\mu_{t,r}(A) := \ProbSupProc_\mu(X_t\in A| X_{t+r}\neq \NullMsr), \quad t,r\in \RealLine_+, A \in \Borel(\FntMsrSp).
\]
Our first result is about the convergence of $\bfQ^\mu_{t,r}$
as $t\to\infty$ with $r$ fixed.
\begin{theorem} \label{thm:Y}
	For any $r\in \RealLine_+$, there exists a probability measure $\bfQ_{\infty,r}$ on $\FntMsrSp$ such that, for any $\mu \in \FntMsrSp^o$, $\bfQ^\mu_{t,r}$ converges weakly to $\bfQ_{\infty,r}$ as $t\to \infty$.
\end{theorem}

	Notice that the case $r=0$ of Theorem \ref{thm:Y} was given in \cite{MR4175472}.

Our next result is about the convergence of $\bfQ^\mu_{t,r}$ as $r\to\infty$ with $t$ fixed.
For any given measurable space $(\tilde \Omega,\tilde {\mathscr F})$, we say a sequence of probability measures
$(\mu_n)_{n=1}^\infty$ on $(\tilde \Omega,\tilde {\mathscr F})$
converge strongly (or converge setwise) to a probability measure $\mu$ on
 $(\tilde \Omega,\tilde {\mathscr F})$
if $\lim_{n\to\infty}\mu_n(f) = \mu(f)$ for any $f\in \Bndd\tilde {\mathscr F}$.  An equivalent definition can be found in \cite{MR1974383}*{Definition 1.4.1}.
\begin{theorem} \label{thm:Q}
	For any $\mu\in \FntMsrSp^o$ and $t\in \RealLine_+$, there exists a probability measure $\bfQ^\mu_{t,\infty}$ on $\FntMsrSp$ such that $\bfQ^\mu_{t,r}$
converges strongly to $\bfQ^\mu_{t,\infty}$ as $r\to \infty$.
\end{theorem}

We then consider the limits of $\bfQ_{\infty,r}$ and $\bfQ^\mu_{t,\infty}$ as $r\to\infty$ and $t\to\infty$ respectively.  Define $\mathcal E \in [0,\infty]$ by
\begin{equation}\label{eq:I.35}
	\mathcal E
	:= \int_E \nu(\Diff x)\int_{(0,\infty)} u\phi(x)\log^+\big(u\phi(x)\big) \pi(x, \Diff u),
\end{equation}
where
$\log^+ z := \max(\log z,0 )$ for every $z>0$.

\begin{theorem} \label{thm:E}
Let $\mu \in \FntMsrSp^o$ be arbitrary.
	The following five statements are equivalent:
\begin{enumerate}
\item
	$\mathcal E < \infty$.
\item
	$\int_\FntMsrSp \eta(\phi) \bfQ_{\infty,0}(\Diff \eta) < \infty$.
\item \label{eq:liminf}
	$\liminf_{t\to \infty} e^{-\lambda t}\ProbSupProc_\mu(X_t \neq \NullMsr) > 0$.
\item \label{eq:StrongDoubleLimit}
	$\bfQ_{\infty, r}$ converges strongly as $r\to \infty$.
\item \label{eq:WeakDoubleLimit}
	$\bfQ^\mu_{t,\infty}$ converges weakly as $t\to \infty$.
\end{enumerate}
\end{theorem}

	Theorem \ref{thm:E} can be considered as an analog Heathcote, Seneta and Vere-Jones' $L\log L$ theorem for superprocesses.
	In particular, the condition $\mathcal E < \infty$ is an
	analog of the $L \log L$ condition \eqref{eq:I.01}.
	The same condition has already appeared in  \cite{MR2535827}
	where the first three authors of this paper studied the asymptotic behavior of supercritical superdiffusions.
	Here, in the subcritical setting,  $\mathcal E< \infty$ is shown to be equivalent to the exponential decay of the survival probability.
	(We are using `liminf' in the third statement because we are not assuming prior existence of the limit.
	In fact, it is made clear in the the next theorem that the limit does exist under the condition $\mathcal E < \infty$.)
	The equivalence of	$\mathcal E< \infty$ to the existence of the two types of double limits in \eqref{eq:StrongConcide} and \eqref{eq:WeakConside} is
	in parallel to Pakes' double limit theorem \cite{MR1735780}*{Theorems 2.2 and 2.3}.
	Notice that in \eqref{eq:StrongDoubleLimit} and \eqref{eq:WeakDoubleLimit}, the sense of convergence for these two double limits are different.
	This differences is not
         present in the context of Galton-Watson processes
	because the weak convergence and the strong convergence are equivalent for the probability distributions on the discrete space $\mathbb N$.
	The following theorem says that
	the two double limits coincide,
	which is in parallel to Pakes' result on the Galton-Watson branching processes.
		We also gives the weak limit for $\bfQ^\mu_{t,r}$ when $t$ and $r$ converges to $\infty$ together.
		It seems that this latter result has not been explored before for other Markov branching processes.

\begin{theorem} \label{thm:E.0}
	Suppose that $\mathcal E<\infty$. Then there exists a probability measure $\bfQ_{\infty,\infty}$ on $\FntMsrSp$ such that the following four statements hold for any  $\mu\in \FntMsrSp^o$:
\begin{enumerate}
\item
	$\bfQ_{\infty,\infty}(\Diff \eta) =\eta(\phi) \bfQ_{\infty,0}(\Diff \eta)/ \int_\FntMsrSp \eta(\phi) \bfQ_{\infty,0}(\Diff \eta).$
\item
	$\lim_{t\to \infty} e^{-\lambda t}\ProbSupProc_\mu(X_t \neq \NullMsr) = \mu(\phi)/ \int_\FntMsrSp \eta(\phi) \bfQ_{\infty,0}(\Diff \eta).$
\item \label{eq:StrongConcide}
	$\bfQ_{\infty, r}$ converges strongly to $\bfQ_{\infty,\infty}$ as $r\to \infty$.
\item \label{eq:WeakConside}
	$\bfQ^\mu_{t,\infty}$ converges weakly to $\bfQ_{\infty,\infty}$ as $t\to \infty$.
	\item
	$\bfQ^\mu_{t,r}$ converges weakly to $\bfQ_{\infty,\infty}$ as $t,r\to \infty$.
\end{enumerate}
\end{theorem}

\begin{remark}
	If the space $E$ only contains one point, i.e. $E = \{x\}$, the superprocess $X$ degenerates to a continuous-state branching process. In this special case, Assumptions \eqref{asp:H1} and
\eqref{asp:H2}
hold automatically, and the main results of this paper have already been given by \cites{MR408016,MR2299923} and \cite{MR1727226}.
\end{remark}
	\begin{remark}
		If $E = \{x_1,\cdots, x_n\}$ is a finite set and the $E$-valued Markov chain $(\xi_t)_{t\geq 0}$ is irreducible, then the superprocess $X$ degenerates to an irreducible multitype continuous state branching process.
		 In this case, one can verify from the Perron-Frobenius theory that the Assumptions \eqref{asp:H1} and
\eqref{asp:H2} hold.
		If one further assume that the kernel $\pi(x,\mathrm du)= 0$, then our results \eqref{eq:StrongConcide} and \eqref{eq:WeakConside} of Theorem \ref{thm:E.0} were already appeared in \cite{MR2399296}*{Theorem 3.7}.
	\end{remark}

\begin{remark}
	When the branching mechanism $\psi$ is spatially homogeneous,
our Theorem \ref{thm:Q} is an immediate corollary of \cite{MR2760602}*{Theorem 6.8}.
\end{remark}

\subsection*{Overview of the method}
Note that the main results Theorems \ref{thm:Y}-\ref{thm:E.0}
depend only on the transition semigroup $(Q_t)_{t\geq 0}$ of the superprocess $(X_t)_{t\geq 0}$.
Therefore, we can work on any specific realization of $(X_t)_{t\geq 0}$ without loss of generality.
	According to Lemma \ref{thm:A.05}, $(Q_t)_{t\geq 0}$ is a Borel  semigroup on $\FntMsrSp$.
	This and \cite{MR2760602}*{Theorem A.33} allow us to realize the superprocess on the space of $\FntMsrSp$-valued right continuous paths.
	To be more precise, we can, and will, assume the following statements hold throughout the rest of the paper.
\begin{statement}
		$\Omega$ is the space of $\FntMsrSp$-valued right continuous functions on $\RealLine_+$.
\end{statement}
\begin{statement}
			$(X_t)_{t\geq 0}$ is the coordinate process of the path space $\Omega$.
\end{statement}
\begin{statement}
		$(\theta_t)_{t\geq 0}$ are the shift operators on the path space $\Omega$.
\end{statement}
\begin{statement}
		$\mathscr F_t = \sigma(X_s:s\in[0,t])$ and $\mathscr F = \sigma(X_s:s\in \RealLine_+)$.
\end{statement}
\begin{statement}
		For any $\mu\in \FntMsrSp$, $\ProbSupProc_\mu$ is the probability measure on $(\Omega,\srF)$
		so that under $\ProbSupProc_\mu$,
		$X_0 = \mu$ almost surely and that $(X_t)_{t\geq 0}$ is a Markov process with transition semigroup $(Q_t)_{t\geq 0}$.
\end{statement}
	Note that for any $H\in \Bndd\srF$, $\mu \mapsto \ProbSupProc_\mu(H)$ is a measurable function on $\FntMsrSp$.
For any probability measure $\bfP$ on $\FntMsrSp$, we define a probability measure $\bfP\ProbSupProc$ on $(\Omega, \srF)$ by
\[
	(\bfP\ProbSupProc)(A):= \int_\FntMsrSp \ProbSupProc_\mu(A) \bfP(\Diff \mu), \quad A\in \srF.
\]
Denote by $(\srF^\mathrm a, (\srF^\mathrm a_t)_{t\geq 0})$ the augmentation of $(\srF,(\srF_t)_{t\geq 0})$ by the system of probability measures $\{\bfP\ProbSupProc: \bfP \text{ is a probability measure on } \FntMsrSp\}$.
Then, according to \cite{MR2760602}*{Lemma A.33},
\[
X:=(\Omega, \srF^\mathrm a, (\srF^\mathrm a_t)_{t\geq 0}, (X_t)_{t\geq 0}, (\theta_t)_{t\geq 0}, (\mathrm P_\mu)_{\mu\in \FntMsrSp} )
\]
is a Borel right process with transition semigroup $(Q_t)_{t\geq 0}$, i.e., a $(\xi,\psi)$-superprocess.

	We already proved Theorem \ref{thm:Y} in the case $r=0$ in \cite{MR4175472}.
	For the case $r>0$, we will give a stronger result by considering the shifted
	two-sided process $(X_{t+u})_{u\in \mathbb R}$ with the convention  $X_s := \NullMsr$ for $s<0$.
	We will show in Proposition \ref{thm:Y.2} that this two-sided process,
conditioned on $\{X_t \neq \NullMsr\}$, has a limiting process $(Y_u)_{u\in \mathbb R}$ when $t\to\infty$.
	We will obtain this result by analyzing the Laplace transform of the shifted
	two-sided process.

	We will also establish a stronger version of Theorem \ref{thm:Q} by considering the (non-shifted) process
 $(X_u)_{u\geq 0}$ under the condition $\{X_t \neq \NullMsr\}$.
	We will show in Proposition \ref{thm:Q.2} that this process has a limiting process $(\widetilde X_u)_{u\geq 0}$ when $t\to\infty$.
	We obtain this stronger result by a martingale change of measure method.
	The limiting process $(\widetilde X_u)_{u\geq 0}$ is interpreted as a superprocess conditioned on living forever, and is referred to as the Q-process.
	We mention here that the Q-process
	$(\widetilde X_{u})_{u\geq 0}$
	has a different law compared to the process $(Y_t)_{t\geq 0}$ above.
	This Q-process also arises in another type of conditioning, see \cite{MR3035765}.
	The study of the Q-process can be traced back to Lamperti and Ney \cite{MR0228073} where they considered the Q-process for
Galton-Watson processes.
For studies on the Q-processes of other models, we refer our readers
 to \cites{MR2994898, penisson2010conditional} and the references therein.

		For the proofs of Theorems \ref{thm:E} and \ref{thm:E.0},
	we use the spine decomposition theorem for superprocesses.
	Roughly speaking, the Q-process
$(\widetilde X_{u})_{u\geq 0}$
can be decomposed in terms of an immortal particle which moves according to a Markov process and generates pieces of mass evolving according to the law of the unconditioned superprocess.
	This representation for the superprocesses was first obtained by \cite{MR1249698}, and developed and generalized into
the spine decomposition theorem
by \cites{MR3395469,MR2040776,MR2006204,MR2535827,MR4058118,MR4359790}.
	Under Assumption \eqref{asp:H2}, this immortal particle will converge in law to its ergodic equilibrium, and the quantitative information about the Q-process can be obtained using the ergodic theorem.
	
	Our proofs of Theorems \ref{thm:E} and \ref{thm:E.0} adopt a method which can be traced back to Lyons, Pemantle and Peres \cite{MR1349164} where they gave a probabilistic proof	of Heathcote, Seneta and Vere-Jones' $L\log L$ theorem for Galton-Watson processes.
	Let us give some intuition here.
	Note that for the spine decomposition of the Q-process, each piece of mass being generated will vanish eventually since they are subcritical  and non-persistent.
	When the $L\log L$ condition holds, the rate at which masses are created is smaller than the rate at which masses vanish,
and the Q-process will converge to an equilibrium state.
	When the $L\log L$ condition does not hold, the rate at which masses are created is bigger than the rate at which they vanish,
	and the Q-process will not converge to any equilibrium because it accumulates more and more mass.

\subsection*{Organization of the paper}
In Section \ref{sec:Y}, we give the proof of Theorem \ref{thm:Y}.
In Section \ref{sec:Q}, we give the proof of Theorem \ref{thm:Q}.
In Section \ref{sec:E}, we give the proofs of Theorems \ref{thm:E} and \ref{thm:E.0},
and the spine decomposition theorems used
are summarized in Lemmas \ref{thm:K.6} and \ref{thm:T.5}.
In Appendix, we gather the proofs of several technical lemmas.

\section{Proof of Theorem \ref{thm:Y}} \label{sec:Y}
We first
recall
some basic results from \cite{MR4175472}.
Define
\begin{equation} \label{eq:M.24}
v_t(x) := -\log \ProbSupProc_{\delta_x}(X_t = \NullMsr), \quad t\geq 0,x\in E.
\end{equation}
From \eqref{eq:I.13} and the monotone convergence theorem, we get that
\begin{equation} \label{eq:M.25}
	\mu(v_t)
	= -\log \ProbSupProc_\mu(X_t = \NullMsr),
	\quad \mu\in \FntMsrSp, t\geq 0.
\end{equation}
In particular, from \eqref{asp:H3}, we have that $\nu(v_t)<\infty$ for $t>0$.
The following lemma, which is a corollary of \cite{MR4175472}*{Proposition 2.2},
entails that $  \{v_t:t>0\}\subset \BnddPos\Borel(E).$
\begin{lemma} \label{thm:M.3}
	For any $t>0$ and $x\in E$, $v_t(x) = \phi(x)\nu(v_t) (1+\Cr{const:M.3}(t,x))$, where
	$\C\label{const:M.3}(t,x)\in \RealLine$
	satisfies that
	$\lim_{t\to\infty}\sup_{x\in E}|\Cr{const:M.3}(t,x)| = 0$.
\end{lemma}

We will also use the following fundamental fact
for the subcritical superprocess $X$.
It can be verified, for example, using \eqref{eq:M.25}
and \cite{MR4175472}*{(3.39)}.
\begin{lemma} \label{thm:M.2}
	For any $\mu \in \FntMsrSp$, $\lim_{t\to\infty}\ProbSupProc_\mu(X_t = \NullMsr) = 1$.
\end{lemma}
In \cite{MR4175472}, we already showed that there exists a probability measure $\bfQ_{\infty,0}$ on $\FntMsrSp$ such that for every $\mu \in \FntMsrSp^o$, $\bfQ^\mu_{t,0}$ converges weakly to $\bfQ_{\infty,0}$ as $t\to \infty$.
$\bfQ_{\infty,0}$ is known as the Yaglom limit of the superprocess $X$.
It was also proved there that $\mathbf Q_{\infty,0}$ is
the quasi-stationary distribution
for $(X_t)_{t\geq 0}$ with extinction rate $-\lambda$, i.e.,
\begin{equation} \label{eq:Y.05}
	(\bfQ_{\infty,0} \ProbSupProc)(X_r \in \Diff \mu | X_r \neq \NullMsr) = \bfQ_{\infty,0}(\Diff \mu)
\end{equation}
and
\begin{equation} \label{eq:Y.1}
	(\bfQ_{\infty,0}\ProbSupProc)(X_r\neq \NullMsr) = e^{\lambda r} > 0.
\end{equation}
\eqref{eq:Y.1} allows us
to define a probability measure $\bfQ_{\infty,r}$ on $\FntMsrSp$ for any $r\geq 0$ such that
\begin{equation} \label{eq:Y.15}
	\bfQ_{\infty,r}[F] = (\bfQ_{\infty,0}\ProbSupProc)[F(X_0)|X_r\neq \NullMsr], \quad F\in \Bndd\Borel(\FntMsrSp).
\end{equation}
We will prove Theorem \ref{thm:Y} by showing that $\bfQ_{\infty, r}$ is
the weak limit of $\bfQ_{t,r}^\mu$ when
$t\to \infty$ for any $\mu\in\mathcal M^o$.
In fact, we can prove a proposition which is stronger than Theorem \ref{thm:Y}.
To formulate this proposition, we first prove a lemma. We will use the convention $X_t := \NullMsr$ for $t<0$.
\begin{lemma}
\label{thm:Y.1}
There exists a two-sided $\FntMsrSp$-valued process $(Y_u)_{u\in \mathbb R}$ on some probability space such that for any $t>0$, the process $(X_{t+u})_{u\geq -t}$ under  $(\bfQ_{\infty,0}\ProbSupProc)(\cdot|X_{t}\neq \NullMsr)$ has the same finite-dimensional distributions as $(Y_u)_{u\geq -t}$.
\end{lemma}

\begin{proof}
	We say $G$ is a finite-dimensional $[0,\infty]$-valued linear functional
	of
	$\FntMsrSp$-valued two-sided paths if the following statement holds.
\begin{statement}\label{eq:Y.2}
		There exist a natural number $n$,
		$\{u_i:i=1,\dots,n\}\subset \RealLine$, and a list of $[0,\infty]$-valued Borel functions $(f_i)_{i=1}^n$ on $E$, such that $ G(w) = \sum_{i=1}^n w_{u_i}(f_i)$ for every $\FntMsrSp$-valued two-sided path $w=(w_u)_{u\in \RealLine}$.
\end{statement}
Fix an arbitrary finite-dimensional $[0,\infty]$-valued linear functional $G$ as above.
For any $s\in \mathbb R$,  define $G_s(w): = G(w_{s+\cdot})$ for any $\FntMsrSp$-valued two-sided path $w$.
Then $G_s$ is also a finite-dimensional $[0,\infty]$-valued linear functional.
Fix a time $s\geq 0$ large enough so that $s+u_i\geq 0$ for every $i=1,\cdots,n$.
Since
$X$ is a time-homogeneous Markov process, using \eqref{eq:Y.05} and \eqref{eq:Y.1}, we have that for any $t\geq s$,
 \begin{align}
	&(\mathbf Q_{\infty,0}\ProbSupProc) (e^{-G_t(X)}|X_t\neq \NullMsr)
	= e^{-\lambda t}\cdot (\mathbf Q_{\infty,0}\ProbSupProc)[e^{-G_t(X)}\Ind_{\{X_t\neq \NullMsr\}} ]
       \\&= e^{-\lambda t}\cdot (\mathbf Q_{\infty,0}\ProbSupProc)\big[ \Ind_{\{X_{t-s}\neq \NullMsr\}}\ProbSupProc_{X_{t-s}}[e^{-G_s(X)}\Ind_{\{X_s\neq \NullMsr\}}]\big]
      \\&=e^{-\lambda s}\cdot (\mathbf Q_{\infty,0}\ProbSupProc)\Big[ \ProbSupProc_{X_{t-s}}[e^{-G_s(X)}\Ind_{\{X_s\neq \NullMsr\}}]\Big|X_{t-s}\neq \NullMsr\Big]
      \\&=e^{-\lambda s}\cdot (\mathbf Q_{\infty,0}\ProbSupProc)\big[ e^{-G_s(X)}\Ind_{\{X_s\neq \NullMsr\}}\big]
      = (\mathbf Q_{\infty,0}\ProbSupProc) (e^{-G_s(X)}|X_s\neq \NullMsr).
\end{align}
In other words,
given a finite subset $U=\{u_i: i=1,\dots,n\}\subset \RealLine$ and a large enough $t\geq 0$, the $\FntMsrSp$-valued random vector $(X_{t+u})_{u\in U}$ under the probability $(\mathbf Q_{\infty,0}\ProbSupProc)(\cdot| X_t\neq \NullMsr)$ has a distribution, denoted by
$\mathcal D_U$, which is independent of the choice of $t$.
Using the Markov property, it is easy to verify that this family of finite-dimensional distributions $\mathcal D:=\{\mathcal D_U: \text{$U$ is a finite subset of $\mathbb R$}\}$ satisfies the consistency condition for the Kolmogorov extension theorem.
Therefore, there exists a two-sided $\mathcal M$-valued process $(Y_u)_{u\in \mathbb R}$ whose finite-dimensional distributions are given by $\mathcal D$.

It is a routine to verify that $(Y_u)_{u\in \mathbb R}$ satisfies the desired properties of this lemma.
\end{proof}

Recall that the two-sided indexed process  $(X_{t})_{t\in \mathbb R}$ is
defined with the convention that  $X_s := \NullMsr$ for $s<0$.

\begin{proposition}
	\label{thm:Y.2}
	For any $\mu\in \FntMsrSp^o$, when $t\to\infty$, the $\FntMsrSp$-valued two-sided	process
$(X_{t+u})_{u\in\RealLine}$  under  $\ProbSupProc_\mu(\cdot | X_t\neq \NullMsr)$ converges  to the process $(Y_u)_{u\in \mathbb R}$, given
	in Lemma \ref{thm:Y.1},
in the sense of  finite-dimensional distributions.
\end{proposition}

We first explain that the above proposition is indeed stronger than Theorem \ref{thm:Y}.
\begin{proof}[Proof of Theorem \ref{thm:Y}]
Fix arbitrary $r\geq 0$ and
$F\in \Bndd C(\FntMsrSp)$.
	Using Proposition \ref{thm:Y.2} and Lemma \ref{thm:Y.1}, we have
\begin{align}
	&\bfQ_{t,r}^\mu[F] = \ProbSupProc_\mu[F(X_{(t+r)-r}) |X_{t+r}\neq 0]
     \\&\xrightarrow[t\to \infty]{} \mathbb E[F(Y_{-r})]
      =(\bfQ_{\infty,0}\ProbSupProc)[F(X_0)|X_{r}\neq \NullMsr]
      = \bfQ_{\infty,r}[F]
\end{align}
as desired.
\end{proof}

Before we prove Proposition \ref{thm:Y.2}, we first present the following two lemmas.

\begin{lemma}\label{thm:Y.4}
	For any $\mu \in \FntMsrSp^o$ and $[0,\infty]$-valued Borel function $f$ on $E$,
\[
	\int_\FntMsrSp e^{-\eta(f)}\bfQ^\mu_{t,0}(\Diff \eta)
	\xrightarrow[t\to\infty]{} \int_\FntMsrSp e^{-\eta(f)} \bfQ_{\infty,0}(\Diff\eta).
\]
\end{lemma}

Lemma \ref{thm:Y.4} follows from \cite{MR4175472}*{Proposition 2.3 \& (2.9)}.

\begin{lemma}\label{thm:Y.3}
For any $\eta,\mu\in \FntMsrSp^o$ and $s\in \RealLine_+$, it holds that
\begin{equation}
	\frac{\mathrm P_{\eta}(X_{t-s} \neq \NullMsr)}{\mathrm P_\mu(X_{t} \neq \NullMsr)}
	\xrightarrow[t\to \infty]{} \frac{e^{-\lambda s}\eta(\phi)}{\mu(\phi) }.
\end{equation}
\end{lemma}
\begin{proof}
	Let $v_t(x)$ be given as in \eqref{eq:M.24}.
We have from Lemma \ref{thm:M.3}, \cite{MR4175472}*{(3.20)}
and the bounded convergence theorem that, for any $\eta, \mu\in \FntMsrSp^o$ and
$s\geq 0$,
	\begin{align}
&\lim_{t\to \infty}\dfrac{ \eta(v_{t-s}) }{ \mu(v_t) }
=\lim_{t\to \infty} \frac{\nu(v_{t-s}) \int \phi(x)(1+\Cr{const:M.3}(t-s,x))\eta(\Diff x)}{\nu(v_t) \int \phi(x)(1+\Cr{const:M.3}(t,x))\mu(\Diff x)}
= \frac{e^{-\lambda s} \eta(\phi)}{\mu(\phi)}.
	\end{align}
Thus we have by \eqref{eq:M.25} and \cite{MR4175472}*{(3.39)} that,
	\begin{equation}
\lim_{t\to\infty}\frac{\mathrm P_{\eta}(X_{t-s}\neq \NullMsr)}{\mathrm P_\mu(X_t\neq \NullMsr)}
=\lim_{t\to\infty}\frac{1-e^{- \eta(v_{t-s}) }}{1-e^{- \mu(v_t) }}
=\lim_{t\rightarrow\infty}\dfrac{ \eta(v_{t-s}) }{ \mu(v_t) }
= \frac{e^{-\lambda s} \eta(\phi)}{\mu(\phi)}.
		\qedhere
	\end{equation}
\end{proof}

\begin{proof}[Proof of Proposition \ref{thm:Y.2}]
	To prove the convergence of the processes, we verify the convergence of all the Laplace transforms of the finite-dimensional linear functional.
Fix an arbitrary $\mu\in \FntMsrSp^o$ and an arbitrary finite-dimensional $[0,\infty]$-valued linear functional
$G$  defined in \eqref{eq:Y.2}.
It can be verified using \eqref{eq:I.13}, the Markov property and induction that
there exists a $[0,\infty]$-valued Borel function $v_G$ on $E$, which depends on the choice of $G$ but not on $\mu$, such that $\ProbSupProc_\mu[\exp\{-G(X)\}] = \exp\{-\mu(v_G)\}$.
Fix a time $s\geq 0$ large enough so that $s+u_i\geq 0$ for every $i=1,\cdots,n$.
	From the Markov property, we can verify that for any $t\geq s$,
	\begin{align}
		&\ProbSupProc_\mu(e^{-G(X_{t+\cdot})}|X_t\neq \NullMsr)
		= \frac{\ProbSupProc_\mu[e^{-G(X_{t+\cdot})}\Ind_{\{X_t\neq \NullMsr\}} ]}{\ProbSupProc_\mu(X_t\neq 0)}
		= \frac{\ProbSupProc_\mu\big[ \ProbSupProc_{X_{t-s}}[e^{-G(X_{s+\cdot})}\Ind_{\{X_s\neq \NullMsr\}}]\big]}{\ProbSupProc_\mu(X_t\neq 0)}
		\\&=\frac{\ProbSupProc_\mu(X_{t-s} \neq \NullMsr)}{\ProbSupProc_\mu(X_t\neq \NullMsr)}\ProbSupProc_\mu\Big[ \ProbSupProc_{X_{t-s}}[e^{-G(X_{s+\cdot})}\Ind_{\{X_s\neq \NullMsr\}}]\Big|X_{t-s}\neq \NullMsr\Big]
		\\ \label{eq:Confusing}&=\frac{\ProbSupProc_\mu(X_{t-s} \neq \NullMsr)}{\ProbSupProc_\mu(X_t\neq \NullMsr)}\ProbSupProc_\mu\Big[ e^{-X_{t-s}(v_{G_s})}-e^{-X_{t-s}(v_{\tilde G_s})} \Big|X_{t-s}\neq \NullMsr\Big]
	\end{align}
	where $G_s$ and $\tilde G_s$ are finite-dimensional $[0,\infty]$-valued linear functional defined so that $G_s (w) = G(w_{s+\cdot})$ and $\tilde G_s(w) = G(w_{s+\cdot}) +  w_s(\infty \mathbf 1_E)$ for any $\FntMsrSp$-valued two-sided path $w$.
	In fact, \eqref{eq:Confusing} holds because for any $\eta \in \FntMsrSp$,
\[
	\ProbSupProc_{\eta}[e^{-G(X_{s+\cdot})}\Ind_{\{X_s\neq \NullMsr\}} ]
	= \ProbSupProc_{\eta}[e^{-G_s(X)} - e^{-\tilde G_s (X)}]
	= e^{-\eta(v_{G_s})} - e^{-\eta(v_{\tilde G_s})}.
\]
Now from Lemmas \ref{thm:Y.1}, \ref{thm:Y.4}, \ref{thm:Y.3} and \eqref{eq:Y.1}, we have that
\begin{align}
		&\lim_{t\to \infty}\ProbSupProc_\mu(e^{-G(X_{t+\cdot})}|X_t\neq \NullMsr)
	= e^{-\lambda s} \int_\FntMsrSp (e^{-\eta(v_{G_s})}-e^{-\eta(v_{\tilde G_s})}) \bfQ_{\infty,0}(\Diff \eta)
	\\&= e^{-\lambda s} \int_\FntMsrSp \ProbSupProc_{\eta}[e^{-G(X_{s+\cdot})}\Ind_{\{X_s\neq \NullMsr\}}] \bfQ_{\infty,0}(\Diff \eta)
	=  (\bfQ_{\infty,0}\ProbSupProc)[e^{-G(X_{s+\cdot})}|X_s\neq \NullMsr]
	=  \mathbb E[e^{-G(Y)}].
\end{align}
This and \cite{MR2760602}*{Theorem 1.18} imply the desired result.
\end{proof}

\section{Proof of Theorem \ref{thm:Q}} \label{sec:Q}

According to \eqref{eq:I.2} and \eqref{asp:H1}, we have $\ProbSupProc_\mu[X_t(\phi)] = e^{\lambda t}\mu(\phi)\in (0,\infty)$ for any $t\in\RealLine_+$ and $\mu\in \FntMsrSp^o$.
This allows us to define, for any $t\in \RealLine_+$ and $\mu\in \FntMsrSp^o$, a probability measure $\bfQ^\mu_{t,\infty}$ on $(\FntMsrSp, \Borel(\FntMsrSp))$ such that
\[
	\bfQ^\mu_{t,\infty}[F]
	= \ProbSupProc_\mu\Big[\frac{X_t(\phi)}{e^{\lambda t} \mu(\phi)}\cdot F(X_t) \Big],
	\quad F\in \Bndd\Borel(\FntMsrSp).
\]
We will prove Theorem \ref{thm:Q} by showing that $\bfQ^\mu_{t,\infty}$ is
the strong limit of $\bfQ_{t,r}^\mu$ when $r\to \infty$.

In fact, we can prove a result which is stronger than Theorem \ref{thm:Q}.  Before presenting this result, we introduce some notation and give a technical lemma. Denote by $\ProbQProc^{(t)}_\mu$ the law of the $\FntMsrSp$-valued process $(X_r)_{r\geq 0}$ under
$\ProbSupProc_\mu(\cdot|X_t\neq \NullMsr)$.
More precisely, for any $t\geq 0$ and $\mu\in \FntMsrSp^o$, define  $\ProbQProc^{(t)}_\mu$ as the probability measure on $\Omega$ such that
\[
	\ProbQProc^{(t)}_\mu[H]
	= \ProbSupProc_\mu[H |X_t\neq \NullMsr ],
	\quad H\in \BnddPos\srF.
\]
The following lemma can be verified from \cite{MR958914}*{Theorem 62.19}.

\begin{lemma} \label{thm:Q.1}
	For any $\mu \in \FntMsrSp^o$, there exists a unique probability measure $\ProbQProc^{(\infty)}_\mu $ on $(\Omega, \srF)$ such that for any $s \geq 0$ and
		$H\in \Pos\srF_{s}$,
	it holds that
	\[
		\ProbQProc_\mu^{(\infty)}[H] = \ProbSupProc_\mu\Big[\frac{X_{s}(\phi)}{e^{\lambda s}\mu(\phi)}\cdot H\Big].
	\]
\end{lemma}

We say a family  $(\mathrm R^{(t)})_{t\geq 0}$ of probability measures
on $\Omega$ converges, as $t\to\infty$,  locally strongly to
a probability measure $\mathrm R$ on $\Omega$ if for any $s\geq 0$ and
 $H\in \BnddPos\srF_{s}$
 it holds that $\lim_{t\to\infty}\mathrm R^{(t)}(H)=\mathrm R(H)$.
The following proposition is the main result of this section, and it is stronger than Theorem \ref{thm:Q}.

\begin{proposition}
	\label{thm:Q.2}
	For any $\mu\in \FntMsrSp^o$, $\ProbQProc^{(t)}_{\mu}$ converges to $\ProbQProc^{(\infty)}_\mu$
locally strongly as
	$t\to \infty$.
\end{proposition}

\begin{proof}[Proof of Theorem \ref{thm:Q}]
	We can verify using Lemma \ref{thm:Q.1} and Proposition \ref{thm:Q.2}
	that for any $t\in \RealLine_+$, $\mu\in \FntMsrSp^o$ and
$F\in \Bndd\Borel(\FntMsrSp)$,
\begin{equation}
\bfQ^\mu_{t,r}[F]= \ProbQProc_\mu^{(t+r)}[F(X_t)] \xrightarrow[r\to\infty]{} \ProbQProc_\mu^{(\infty)}[F(X_t)] = \bfQ_{t,\infty}^\mu[F].
      \qedhere
\end{equation}
\end{proof}

Before proving Proposition \ref{thm:Q.2}, we first prove the following lemma.

\begin{lemma}\label{thm:Q.3}
For any $\mu\in \FntMsrSp^o$ and $s\in \RealLine_+$, it holds that
\begin{equation}
	\limsup_{t\to\infty}\sup_{\eta\in \FntMsrSp^o}\frac{1}{\eta(\phi)}\frac{\mathrm P_{\eta}(X_{t-s} \neq \NullMsr)}{\mathrm P_\mu(X_{t} \neq \NullMsr)} < \infty.
\end{equation}
\end{lemma}

\begin{proof}
	Let $v_t(x)$ be given as in \eqref{eq:M.24}.
	By \eqref{eq:M.25} and Lemma \ref{thm:M.3}, we have for any $\mu,\eta\in \FntMsrSp^o$, $s\geq 0$ and $t> s$,
	\begin{equation}
		 \frac{\mathrm P_{\eta}(X_{t-s}\neq \NullMsr)}{\mathrm P_\mu(X_t\neq \NullMsr)}
		\leq \frac{\eta(v_{t-s}) }{\mu(v_t)} \frac{\mu(v_t)}{1-e^{- \mu(v_t) }}
		\leq \frac{\nu(v_{t-s}) }{\nu(v_t)} \frac{(1+ \sup_{x\in E}|\Cr{const:M.3}(t-s,x)|) \eta(\phi)}{\int (1+ \Cr{const:M.3}(t,x))\phi(x) \mu(\Diff x)}\frac{\mu(v_t)}{1-e^{- \mu(v_t) }}.
	\end{equation}
Using this, \cite{MR4175472}*{(3.20)}
 Lemma \ref{thm:M.3}, the bounded convergence theorem and
\cite{MR4175472}*{(3.39)}
\begin{equation}
	\sup_{\eta\in \FntMsrSp^o}\frac{1}{\eta(\phi)}\frac{\mathrm P_{\eta}(X_{t-s} \neq \NullMsr)}{\mathrm P_\mu(X_{t} \neq \NullMsr)}
\leq \frac{\nu(v_{t-s}) }{\nu(v_t)} \frac{(1+ \sup_{x\in E}|C_{t-s,x}|) }{\int (1+ C_{t,x})\phi(x) \mu(\Diff x) }\frac{\mu(v_t)}{1-e^{- \mu(v_t) }}
\xrightarrow[t\to\infty]{} \frac{e^{-s\lambda}}{\mu(\phi)}
\end{equation}
which implies the desired result.
\end{proof}

\begin{proof}[Proof of Proposition \ref{thm:Q.2}]
	Fix arbitrary $\mu\in \FntMsrSp^o$, $s\geq 0$ and
	$H\in \Bndd\srF_{s}$.
It follows from Lemma \ref{thm:Q.3} that
there exist $\C\label{const:Y.5}(\mu,s)>0$ and $t_0>s$
 such that for any $t\geq t_0$, $\ProbSupProc_\mu$-almost surely,
	\begin{equation}\label{eq:Y.5}
	\frac{\ProbSupProc_{X_s}(X_{t-s} \neq \NullMsr)}{\ProbSupProc_\mu(X_{t} \neq \NullMsr)}
	\leq \Cr{const:Y.5}(\mu,s) X_s(\phi).
\end{equation}
Using the Markov property, Lemma \ref{thm:Y.3}, \eqref{eq:Y.5} and the
dominated convergence theorem, we have
	\begin{align}
		&\ProbQProc^{(t)}_{\mu}[H]
		 = \frac{\ProbSupProc_\mu[ H \cdot \Ind_{\{X_t \neq \NullMsr\}}]}{\ProbSupProc_\mu(X_{t}\neq \NullMsr)}
	   	= \ProbSupProc_\mu\Big[ H \cdot \frac{\ProbSupProc_{X_s}(X_{t-s}\neq\NullMsr)}{\ProbSupProc_\mu(X_{t}\neq \NullMsr)} \Big]
	      \\&\xrightarrow[t\to \infty]{} \ProbSupProc_\mu\Big[H \cdot \frac{e^{-\lambda s} X_s(\phi)}{\mu(\phi)}\Big]
	      = \ProbQProc^{(\infty)}_\mu[H]
	\end{align}
	as desired.
\end{proof}

\section{ Proofs of Theorems \ref{thm:E} and \ref{thm:E.0}} \label{sec:E}
Our proofs of Theorems \ref{thm:E} and \ref{thm:E.0} are separated into the following five propositions whose proofs are postponed to Subsections \ref{sec:K}, \ref{sec:L}, \ref{sec:M}, \ref{sec:T}, and \ref{sec:R}, respectively.

\begin{proposition}  \label{thm:K}
	There exists a constant $\mathcal K\in [0,\infty)$ which is independent of the initial value $\mu \in \FntMsrSp^o$ such that
	\begin{equation}\label{eq:E.1}
		\lim_{t\rightarrow\infty} e^{-\lambda t}\mathrm P_\mu(X_t \neq \NullMsr)
		=\mathcal K\mu(\phi), \quad \mu \in \FntMsrSp^o.
	\end{equation}
\end{proposition}
In the remainder of this paper, $\mathcal K$  always denotes the constant above.
\begin{proposition}\label{thm:L}
$\mathcal K>0$ if and only if $\mathcal E<\infty$.
\end{proposition}

\begin{proposition} \label{thm:M}
It holds that
\[
	\int_{\FntMsrSp}\eta(\phi)\bfQ_{\infty,0}(\Diff \eta) = \mathcal K^{-1}.
\]
\end{proposition}
When $\mathcal K>0$, Proposition \ref{thm:M} allows us to consider the (unique) probability measure $\bfQ_{\infty,\infty}$ on $\FntMsrSp$ satisfying
\begin{equation}
		\label{eq:E.11}
	\bfQ_{\infty,\infty} (F)= \int_{\FntMsrSp} F(\eta)\cdot \mathcal K\eta(\phi)\bfQ_{\infty,0}(\Diff \eta),\quad F \in \Bndd \Borel(\FntMsrSp).
\end{equation}

\begin{proposition} \label{thm:T}
	Let $\mu\in\FntMsrSp^o$ be arbitrary.
	If $\mathcal K>0$, then $\bfQ^\mu_{t,\infty}$ converges weakly to $\bfQ_{\infty,\infty}$ as $t\to\infty$.
	If $\mathcal K=0$, then $\bfQ^\mu_{t,\infty}$ does not converge weakly as $t\to\infty$.
\end{proposition}
\begin{proposition} \label{thm:R}
	If $\mathcal K>0$, then $\bfQ_{\infty,r}$ converges
	strongly to $\bfQ_{\infty,\infty}$ as $r\to\infty$.
	If $\mathcal K=0$, then $\bfQ_{\infty,r}$ does not converge
	strongly as $r\to\infty$.
\end{proposition}

\begin{proposition} \label{prop:new}
        If $\mathcal E<\infty$,
   then for any $\mu\in\FntMsrSp^o$
        and non-negative continuous function $f$ on $E$,
	\[
	\lim_{t,r\to\infty}\ProbSupProc_\mu\left(e^{-X_t(f)}\middle|X_{t+r}\neq 0\right)=\int e^{-\eta(f)}
	\bfQ_{\infty,\infty}(\mathrm d \eta).
	\]
\end{proposition}

\begin{proof}[Proofs of Theorems \ref{thm:E} and \ref{thm:E.0}]
	The desired results can be verified directly from Propositions \ref{thm:K}--\ref{prop:new}.
\end{proof}

\subsection{Proof of Proposition \ref{thm:K}} \label{sec:K}

In order to prove Proposition \ref{thm:K}, we need the
spine decomposition theorem for superprocesses.	
To formulate this theorem, we first introduce the \emph{Kuznetsov measures} via the following lemma which is proved in \cite{MR2760602}*{Section 8.4}.
\begin{lemma} \label{thm:K.1}
There exists a unique $\sigma$-finite kernel $\NMsr=(\NMsr_x(A):x\in E, A\in \srF)$ \\
 from $(E,\Borel (E))$ to $(\Omega, \srF)$ such that
\begin{enumerate}
\item
	$\NMsr_x(X_0 \neq \NullMsr) = 0$ for any $x\in E$;
\item
	$\NMsr_x ( X_t =\NullMsr \text{ for all }t\geq 0) =0$
	for every $x\in E$; and
\item
	for any $\mu \in \FntMsrSp$, if
	$\mathbf N$
	is a Poisson random measure on $\Omega$ with intensity $\mu\NMsr$, then
	$(\mu\Ind_{\{0\}}(t)+ \mathbf N(X_t)\Ind_{(0,\infty)}(t))_{t\geq 0}$
	is an $\FntMsrSp$-valued stochastic process of the same finite dimensional distributions as a $(\xi,\psi)$-superprocess with initial value $\mu$.
Here
$\mathbf N(X_t)=\int_\Omega X_t(\omega) \mathbf N(\Diff \omega)=\int_\Omega \omega_t \mathbf N(\Diff \omega)$, $t>0$.
\end{enumerate}
\end{lemma}

The family of $\sigma$-finite measures $(\NMsr_x)_{x\in E}$ is known as the Kuznetsov measures of $X$.
Note that those measures are typically not finite.
One way to use them
is to transform them into probability measures.
	Notice that from Lemma \ref{thm:K.1}(3) and Campbell's theorem,
\begin{equation} \label{eq:L.3}
	(\mu\NMsr)[X_t(f)] = \ProbSupProc_\mu[X_t(f)] = \mu(T_tf), \quad \mu\in \FntMsrSp, t>0, f\in
	\BnddPos
	\Borel(E).
\end{equation}
Therefore, for any $\mu\in \FntMsrSp^o$ and $t>0$, there exists a unique probability measure $\widetilde{\mu \NMsr}^{(t)}$ on $(\Omega,\mathscr F)$ such that for any  $H\in \Bndd\srF$, $\widetilde{\mu \NMsr}^{(t)}[H] = (\mu\NMsr)[H\cdot e^{-\lambda t}\mu(\phi)^{-1}X_t(\phi)].$

	Another ingredient for the spine decomposition theorem is the so-called spine process which is
an $E$-valued Markov process with transition semigroup $(S_t)_{t\geq 0}$ on $E$ defined so that
\begin{equation} \label{eq:E.17}
	S_t f(x)
	= e^{-\lambda t} \phi(x)^{-1} T_t (\phi f) (x),
	\quad t\geq 0, f\in \Bndd\Borel(E),x\in E.
\end{equation}
	The following lemma can be verified using \cite{MR958914}*{Theorem 62.19}.
	
\begin{lemma} \label{thm:K.3}
$(S_t)_{t\geq 0}$ is a conservative Borel right semigroup on $E$.
\end{lemma}

In this section, we will add a little twist to the classical spine decomposition theorem
by only considering a specific initial value $\nu$,  but with a two-sided spine.
This is possible thanks to the following lemma whose proof is
 postponed to the Appendix.
For any probability measure $\mu$ on $E$,
we define a probability measure $\tilde \mu$ on $E$ so that
$$
\tilde \mu(f) = \mu(\phi)^{-1}\mu(\phi f)\quad\mbox{ for every }f\in \BnddPos\Borel(E).
$$
In particular,  $\tilde \nu(f) = \nu(\phi f)$ for any $f\in \BnddPos\Borel(E).$
We say an $E$-valued two-sided process $(g_t)_{t\in \RealLine}$ defined on a probability space $(\Omega_0,\mathscr G)$ is measurable if $(t,\omega)\mapsto g_t(\omega)$ is a measurable map from $(\mathbb R\times \Omega_0, \Borel(\mathbb R)\otimes \mathscr G )$ to $(E,\Borel(E))$.

\begin{lemma} \label{thm:K.4}
$\tilde \nu$ is an invariant probability measure of the semigroup $(S_t)_{t\geq 0}.$
In particular, there exists a two-sided $E$-valued measurable
stationary Markov process with transition semigroup $(S_t)_{t\geq 0}$ and one-dimensional distribution
$\tilde \nu$.
\end{lemma}
\begin{proof}
	It is straight-forward to verify that $\tilde \nu$ is an invariant measure for the semigroup $(S_t)_{t\geq 0}$.
Using Kolmogorov's extension theorem, we can construct an $E$-valued two-sided
stationary Markov process $(\xi^*_t)_{t\in \RealLine}$,
canonically on the product space $E^\mathbb R$ with transition semigroup $(S_t)_{t\geq 0}$ and one-dimensional distribution
$\tilde\nu$.

To finish the proof, we only have to construct a measurable process $(\Spn_t)_{t\in \RealLine}$ which is a modification of $(\xi^*_t)_{t\in \RealLine}$.
To do this, we consider the compact metric space $\tilde E$ which is the Ray-Knight completion of $E$ with respect to the right semigroup $(S_t)_{t\geq 0}$.
(We refer our reader to \cite{MR2760602}*{P.~318} for the precise construction.)
Denote by $\rho$ the corresponding metric of $\tilde E$.
Thanks to \cite{MR2760602}*{Theorem A.30} and Theorem \ref{thm:K.3},
we have $E\in \Borel(\tilde E,\rho)$ and $\Borel(E) = \Borel(E,\rho)$; and therefore,
$(\xi^*_t)_{t\in\RealLine}$ is also an $\tilde E$-valued process.
According to \cite{MR2760602}*{Theorem A.32 \& Proposition A.7}
for any natural number $n$, the $\tilde E$-valued process $(\xi^*_t)_{t\in [-n,\infty)}$
admits an $\tilde E$-\Cadlag modification. Thus $(\xi^*_t)_{t\in\RealLine}$ admits an $\tilde E$-\Cadlag modification,
denoted by $(\xi^{**}_t)_{t\in \RealLine}$.
Finally, fixing an element $x_0\in E$, taking the measurable map $\psi: x\mapsto x\Ind_{x\in E} + x_0 \Ind_{x\in \tilde E\setminus E}$ from $(\tilde E,\Borel(\tilde E))$ to $(E, \Borel(E))$, we can verify that $\tilde \xi_t:= \psi( \xi^{**}_t),t\in \RealLine$ is an $E$-valued measurable modification of the process $(\xi^*_t)_{t\in \RealLine}$ as desired.
\end{proof}

Roughly speaking, the spine decomposition theorem says that the $\FntMsrSp$-valued process $(X_t)_{t\geq 0}$ under the probability $\ProbQProc^{(\infty)}_\mu$ can be decomposed in law as the sum of a copy of the original $(\xi,\psi)$-superprocess and an $\FntMsrSp$-valued immigration process along the trajectory of an immortal moving particle.
Note that we will only consider the case when $\mu$ is taken as $\nu$ in this section.
To formulate this theorem, we construct random elements
$\big(W^{(0)}, \Spn, \mathcal N, (s_k, y_k,W^{(k)})_{k=1}^\infty\big)$,
on a probability space with probability measure $\ProbSpnDec$, so that the following statements \eqref{eq:K.41}--\eqref{eq:K.45} hold.
\begin{statement} \label{eq:K.41}
$\Spn = (\Spn_t)_{t\in \RealLine}$ is a two-sided $E$-valued measurable stationary Markov process with transition semigroup $(S_t)_{t\geq 0}$ and one-dimensional distribution
$\tilde \nu$.
\end{statement}
\begin{statement} \label{eq:K.42}
	Conditioned on $\Spn$,
	$(s_k, y_k)_{k=1}^\infty$
	is a sequence of
    $\RealLine \times \RealLine_+$-valued random elements such that
	$\mathcal D := \sum_{k=1}^\infty \delta_{(s_k, y_k)}$
is a Poisson random measure on
$\RealLine \times \RealLine_+$
with intensity $\Diff s \otimes y \pi(\Spn_s,\Diff y)$.
\end{statement}
\begin{statement} \label{eq:K.43}
	Conditioned on $\Spn$ and
	$(s_k, y_k)_{k=1}^\infty$, $(W^{(k)})_{k=1}^\infty$
	is a sequence of independent
	$\FntMsrSp$-valued right-continuous
stochastic processes
such that, for each
natural number $k$,
	$W^{(k)} = (W^{(k)}_t)_{t\geq 0}$
	has distribution $\mathrm P_{y\delta_{x}}$ where $y=y_k$ and
$x =\Spn_{s_k}$.
\end{statement}
\begin{statement} \label{eq:K.44}
	Conditioned on $\Spn$, $\mathcal N$ is a Poisson random measure on
$\RealLine\times \Omega$, independent of
$(s_k, y_k, W^{(k)})_{k=1}^\infty$,
with intensity
$2\sigma(\Spn_s)^2\Diff s \otimes \NMsr_{\Spn_s}(\Diff w).$
\end{statement}
\begin{statement}\label{eq:K.45}
$W^{(0)} = (W^{(0)}_t)_{t\geq 0}$ is a $\FntMsrSp$-valued right-continuous
process with law $\mathrm P_\nu$, independent of
$\big(\Spn, \mathcal N, (s_k, y_k, W^{(k)})_{k=1}^\infty\big)$.
\end{statement}

\begin{remark}
	The existence of the above  random elements
	$\big(W^{(0)}, \Spn, \mathcal N, (s_k, y_k,W^{(k)})_{k=1}^\infty\big)$
follows from the existence of the spine process (Lemma \ref{thm:K.4}), the superprocesses (\cite{MR2760602}), and the Poisson random measures (\cite{MR3155252}*{Theorem 2.4}).
	The precise construction of a probability space that carries those structures will be omitted since it is tedious but straightforward.
\end{remark}

	Notice that there are two types of immigration along the spine $\Spn$.  The discrete immigration is given by
	$(W^{(k)})_{k=1}^\infty$, while the continuous immigration is governed by the Poisson random measure
	$\mathcal N$.
	We are interested in the total contributions, at a given time $t$, of all the immigration whose earliest immigrant ancestor is born in a time interval $(a,b]$.
	More precisely, we define, for each $-\infty \leq a < b\leq t <\infty$ and $f\in \BnddPos\Borel(E)$,
\begin{equation} \label{eq:K.51}
	Z_t^{(a,b]}(f)
	:=\sum_{k=1}^\infty
	W^{(k)}_{t-s_k}(f) \Ind_{(a,b]}(s_k)+ \int_{(a,b]\times \Omega} w_{t-s}(f) \mathcal N(\Diff s,\Diff w).
\end{equation}
It can be verified using
Lemma \ref{thm:K.6} below
that when $a>-\infty$, $Z_t^{(a,b]}$ is an $\mathcal M$-valued random element.
However, this does not hold in general if $a = -\infty$.
In particular, $Z_0^{(-\infty,0]}(\phi)$
might take $\infty$ as a value.
With the convention that $\infty^{-1} = 0$ and $0^{-1}=\infty$, we define a constant
\begin{equation} \label{eq:K.52}
	\mathcal K := \ProbSpnDec[Z_0^{(-\infty,0]}(\phi)^{-1}].
\end{equation}
We will prove Proposition \ref{thm:K} by showing that $\mathcal K$ is finite and fulfills \eqref{eq:E.1}.

The  spine decomposition theorem will be summarized in the following lemma.
For its proof, we refer our readers to \cite{MR4058118}*{Theorem 1.5 \& Corollary 1.6}.
We define $Z_t^{(0,0]}=\NullMsr$ for any $t\geq 0.$

\begin{lemma} \label{thm:K.6}
	The $\FntMsrSp$-valued process
	$(W^{(0)}_t+ Z_t^{(0,t]})_{t\geq 0}$
under $\ProbSpnDec$ has
the same finite-dimensional distributions as the coordinate process
	$(X_t)_{t\geq 0}$
under  $\ProbQProc^{(\infty)}_\nu.$
	Moreover, for any $t_0>0$, the $\FntMsrSp$-valued process $(Z_t^{(0,t]})_{t\in [0,t_0]}$
under $\ProbSpnDec$ has
	the same finite-dimensional distributions as the coordinate process $(X_t)_{t\in [0,t_0]}$
	under $\widetilde{\nu\NMsr}^{(t_0)}.$
\end{lemma}

We are now ready to give the proof of Proposition \ref{thm:K}.
\begin{proof}[Proof of Proposition \ref{thm:K}]
\emph{Step 1.}
One can verify that for any $-\infty < a < b \leq t<\infty$ and $s\in \RealLine$, the $\mathcal M$-valued random elements $Z_t^{(a,b]}$ and $Z_{t+s}^{(a+s,b+s]}$ have the same distribution.
This is due to the fact that both the discrete immigration \eqref{eq:K.42}--\eqref{eq:K.43} and the continuous immigration \eqref{eq:K.44} are defined in a time-homogeneous way along the spine $(\Spn_t)_{t\in \RealLine}$ which is a stationary process \eqref{eq:K.41}.
	
\emph{Step 2.}
Let $\mathcal K$ be given as in \eqref{eq:K.52}.
We will show that $\mathcal K<\infty$ and
\eqref{eq:E.1} holds
when $\mu = \nu$.
	The main idea is to work with the reciprocal of the additive martingale $e^{-\lambda t} X_t(\phi)$ under the
measure $\ProbQProc^{(\infty)}_\nu$ to analyze the survival probability.
In fact, for any $t\geq 0$, from Lemmas \ref{thm:Q.1}, \ref{thm:K.6}, and Step 1, we have
\begin{align}
	&e^{-\lambda t}\ProbSupProc_\nu(X_t\neq \NullMsr)
	= \ProbQProc^{(\infty)}_\nu[X_t(\phi)^{-1}]
      \\&=\ProbSpnDec\big[ \big(W^{(0)}_t(\phi) +Z^{(0,t]}_t(\phi)\big)^{-1}\big]
	\label{eq:K.7}
	=\ProbSpnDec\big[ \big(W_t^{(0)}(\phi) +Z^{(-t,0]}_0(\phi)\big)^{-1}\big].
\end{align}
From \eqref{asp:H1},\eqref{eq:K.45}
and the Markov property of superprocesses, we can verify that the process $(e^{-\lambda t} W^{(0)}_t(\phi))_{t\geq 0}$
is a non-negative $\ProbSpnDec$-martingale.
So by the martingale convergence theorem and \eqref{asp:15}, we have
$\ProbSpnDec$-almost surely $W_t^{(0)}(\phi) \to 0$ as $t\to \infty$.
From the fact that $t\mapsto Z^{(-t,0]}_0(\phi)$ is a non-decreasing process with almost sure limit $Z^{(-\infty,0]}_0(\phi)$ in $[0,\infty]$, we have almost surely \[ (W_t^{(0)}(\phi) +Z^{(-t,0]}_0(\phi))^{-1} \xrightarrow[t\to \infty]{} Z^{(-\infty,0]}_0(\phi)^{-1}\in [0,\infty].\]
	Now, we can apply the dominated convergence theorem in \eqref{eq:K.7} and get the desired result in this step.
In fact, the family of non-negative random variables $\{(W_t^{(0)}(\phi) +Z^{(-t,0]}_0(\phi))^{-1}:t\geq 1\}$ is dominated by $Z^{(-1,0]}_0(\phi)^{-1}$, which is $\ProbSpnDec$-integrable since, according to Step 1,
Lemmas \ref{thm:K.6}, \ref{thm:K.1}(3), Campbell's theorem and \eqref{asp:H1}, we have
\begin{align}
	&\ProbSpnDec [Z^{(-1,0]}_0(\phi)^{-1}]
	= \ProbSpnDec [Z^{(0,1]}_1(\phi)^{-1}]
	= \widetilde {\nu \NMsr}^{(1)} [X_1(\phi)^{-1}]
	\\& = e^{-\lambda} \cdot (\nu\NMsr) (X_1\neq \NullMsr)
	= -e^{-\lambda} \log \mathrm P_{\nu}(X_1 = \NullMsr)
	< \infty. \label{eq:K.8}
\end{align}

\emph{Final step.}
To see \eqref{eq:E.1} holds for all $\mu\in \FntMsrSp^o$, we use Step 2 and Lemma \ref{thm:Y.3}.
\end{proof}

\subsection{Proof of Proposition \ref{thm:L}} \label{sec:L}

Let $\big(W^{(0)}, \Spn, \mathcal N, (s_k, y_k,W^{(k)})_{k=1}^\infty\big)$
be the random elements constructed in Section \ref{sec:K}.
Our proof of Proposition \ref{thm:L} will rely on the following two lemmas.
\begin{lemma}
\label{thm:L.1}
 	There exist $s_0, \epsilon, \theta>0$ and $\delta > 0$  such that for any $x\in E, s>s_0$ and $y\geq e^{\epsilon s} / \phi(x)$, it holds that
 $\mathrm P_{y \delta_{x}}(X_{s}(\phi)>\theta) > \delta$.
\end{lemma}

\begin{proof}
From \cite{MR4175472}*{(3.20)} we know that
there exist $t_0,a,\epsilon>0$ such that for all $s\geq t_0$,
we have $\nu(V_{s}\phi) \geq a \exp(-\epsilon s)$.
According to \cite{MR4175472}*{Proposition 2.2} we know that
there exists $s_0'>0$ such that
for all $s\geq s_0'$ and $x\in E$ we have $V_{s}\phi(x)\ge\frac{1}{2}\phi(x)\nu(V_{s}\phi)$.
	Now take $s_0:= t_0 \vee s_0'$, we have for all $s\geq s_0$ and $x\in E$, $V_{s}\phi(x)\geq \dfrac{a}{2}\phi(x)e^{-\epsilon s}$.
	Let $\theta \in (0,a/2)$.
	We have for any $s>s_0$, $x\in E$ and $y\geq \frac{e^{\epsilon s}}{\phi(x)}$ that
\begin{align}
	&\mathrm P_{y\delta_x}\big(w_s(\phi)>\theta\big)
	=\mathrm P_{y\delta_x}\left(e^{-w_s(\phi)}<e^{-\theta}\right)
	\\&=1-\mathrm P_{y\delta_x}(e^{-w_s(\phi)}\geq e^{-\theta})
	\overset{\text{Chebyshev}}\geq 1-e^{\theta}\mathrm P_{y\delta_x}[e^{- w_s(\phi)}]
	\\&=1-e^{\theta}e^{-yV_{s}\phi(x)}
	\geq 1-e^{\theta}e^{-y \frac{a}{2}\phi(x)e^{-\epsilon s}}
	\geq 1 - e^{\theta - a/2}=: \delta >0
\end{align}
as desired.
\end{proof}

\begin{lemma}\label{thm:L.2}	
	\emph{(1)} If $\mathcal E<\infty$, then for any $\epsilon>0$,
\[
	\sum_{k=1}^\infty\Ind_{(-\infty,0]}(s_k) \cdot y_k e^{\epsilon s_k} \cdot \phi(\Spn_{s_k})< \infty,\quad \ProbSpnDec\text{-a.s.}
\]
\emph{(2)}
If  $ \mathcal E=\infty$,
	then for any $\epsilon>0$ and $s_0\geq 0$,
   	\begin{equation}
	\int^{-s_0}_{-\infty} \Diff s\int_{e^{-\epsilon s}\phi(\Spn_s)^{-1}}^\infty y\pi(\Spn_s,\Diff y)=\infty,\quad {\ProbSpnDec}\text{-a.s.}
	\end{equation}
\end{lemma}
Lemma \ref{thm:L.2} is similar to \cite{MR2535827}*{Lemma 3.2} and its proof is pretty long.
We postpone its proof to the Appendix.
\begin{proof}[Proof of Proposition \ref{thm:L}]

\emph{Step 1.}
	Assuming  $\mathcal E<\infty$, we will show that $\mathcal K>0$.
To do this, we verify using Campbell's theorem, \eqref{asp:H1}, \eqref{eq:L.3}, \eqref{eq:K.41} and \eqref{eq:K.44} that
\begin{align}
	&\ProbSpnDec\Big[\int_{(-\infty,0]\times \Omega}w_{-s}(\phi) \mathcal N(\Diff s, \Diff w) \Big]
	= \ProbSpnDec\Big[\int_{-\infty}^0 2 \sigma(\Spn_s)^2 \Diff s \int_\Omega w_{-s}(\phi)\NMsr_{\Spn_s}(\Diff w) \Big]
	\\& \label{eq:L.31} =\ProbSpnDec \Big[\int^0_{-\infty}  2\sigma(\Spn_s)^2 e^{-\lambda s}\phi(\Spn_s) \Diff s\Big]
     = 2\int^0_{-\infty} e^{-\lambda s} \tilde\nu(\sigma^2 \phi) \Diff s
	<\infty,
	\end{align}
	where in the last inequality we used the fact that $\sigma, \phi\in \Bndd\Borel(E)$ and $\lambda < 0$.
	Then we define the $\sigma$-algebra
	$\mathscr G := \sigma(\Spn, (s_k, y_k)_{k = 1}^\infty)$,
	and verify from	\eqref{eq:I.2}, \eqref{asp:H1}, \eqref{eq:K.43}, and
Lemma \ref{thm:L.2}
that $\ProbSpnDec$-almost surely,
\begin{align}
& \ProbSpnDec\Big(\sum_{k=1}^\infty W^{(k)}_{-s_k}(\phi) \Ind_{(-\infty,0]}(s_k)\Big|\mathscr G \Big)
	= \sum_{k=1}^\infty  \mathrm P_{y_k\delta_{\Spn_{s_k}}}[X_{-s_k}(\phi)]\Ind_{(-\infty,0]}(s_k)
     \\ &\label{eq:L.32}= \sum_{k=1}^\infty y_k \cdot (T_{-s_k} \phi)(\Spn_{s_k}) \cdot \Ind_{(-\infty,0]}(s_k)
	= \sum_{k=1}^\infty y_k e^{-\lambda s_k} \phi(\Spn_{s_k}) \cdot \Ind_{(-\infty,0]}(s_k)
	< \infty.
\end{align}
From \eqref{eq:L.31} and \eqref{eq:L.32}, we can verify that $\mathbf Q$-almost surely,
\begin{equation}
	Z_0^{(-\infty, 0]}(\phi)
	= \int_{(-\infty,0]\times \Omega}w_{-s}(\phi) \mathcal N(\Diff s, \Diff w) +
\sum_{k=1}^\infty w^{(k)}_{-s_k}(\phi) \Ind_{(-\infty,0]}(s_k)< \infty.
\end{equation}
It follows from  \eqref{eq:K.52} that
$\mathcal K = \ProbSpnDec[Z^{(-\infty,0]}_0(\phi)^{-1}] >0$
as desired.

\emph{Step 2.}
	Assuming $\mathcal E = \infty$, we will show that $\mathcal K = 0$.
	Let $s_0, \epsilon, \theta$ and $\delta > 0$ be given as in Proposition \ref{thm:L.1}.
	We claim that in this case
\begin{equation}
\label{eq:L.4}
n_\theta:=\#\{k: k\geq 1, k\in \mathbb Z, s_k\leq 0, W^{(k)}_{-s_k}(\phi)>\theta\}
	= \infty,
	\quad \ProbSpnDec\text{-a.s.}
\end{equation}
	Using this claim, we immediately have  that $Z_0^{(-\infty,0]}(\phi)\geq \theta n_\theta  = \infty$ almost surely,
which implies the desired result since $\mathcal K = \ProbSpnDec[Z_0^{(-\infty,0]}(\phi)^{-1}].$
Now we prove the claim \eqref{eq:L.4}.
From \eqref{eq:K.43}, we have $\ProbSpnDec$-almost surely,
\begin{align}
	&\ProbSpnDec[e^{- n_\theta}|\mathscr G]
	=\prod_{k=1}^\infty
	\ProbSpnDec[\exp\{- \Ind_{(-\infty,0]}(s_k)\Ind_{\{W^{(k)}_{-s_k}(\phi)>\theta\}}\}|\mathscr G]
      \\&
      =\prod_{k=1}^\infty
      \ProbSupProc_{y_k\delta_{\Spn_{s_k}}}[\exp\{- \Ind_{(-\infty,0]}(s_k)\Ind_{\{X_{-s_k}(\phi)>\theta\}}\}]
       = \exp\Big\{-
      \int_{\RealLine\times\RealLine_+}
       F(s,y) \mathcal D(\Diff s, \Diff y)\Big\}
\end{align}
where for any
$(s,y)\in \RealLine\times\RealLine_+$,
 the random variable $F(s,y)$ is given by
\[
	F(s,y)
	:=-\log\ProbSupProc_{y\delta_{\Spn_{s}}}[\exp\{- \Ind_{(-\infty,0]}(s)\Ind_{\{X_{-s}(\phi)>\theta\}}\}].
\]
Now by \eqref{eq:K.42} and Campbell's theorem,
\begin{align}
	&\ProbSpnDec[ e^{- n_\theta} | \Spn ]
	= \exp \Big(-\int_\RealLine \Diff s \int_{(0,\infty)}  (1-e^{-F(s,y)}) y \pi(\Spn_s,\Diff y)\Big)
      \\&= \exp \Big(-\int_\RealLine \Diff s \int_{(0,\infty)}  \ProbSupProc_{y\delta_{\Spn_{s}}}[1-\exp\{- \Ind_{(-\infty,0]}(s)\Ind_{\{X_{-s}(\phi)>\theta\}}\}] y \pi(\Spn_s,\Diff y)\Big)
	\\\label{eq:L.5}&= \exp \Big(-(1-e^{-1})\int_{-\infty}^0 \Diff s \int_{(0,\infty)}  \ProbSupProc_{y\delta_{\Spn_{s}}}( X_{-s}(\phi)>\theta ) y \pi(\Spn_s,\Diff y)\Big).
\end{align}
Note that from
Lemmas \ref{thm:L.1}
and \ref{thm:L.2}(2), we have $\ProbSpnDec$-almost surely,
\begin{align}
	&\int_{-\infty}^0 \Diff s \int_0^\infty   \mathrm P_{y \delta_{\Spn_s}}(X_{-s}(\phi)>\theta) y \pi(\xi_s, \Diff y)
	\geq \delta \int_{-\infty}^{-s_0} \Diff s
	\int_{\phi(\Spn_s)^{-1}e^{-\epsilon s}}^\infty  y \pi(\Spn_s, \Diff y)
  	= \infty.
\end{align}
Now from this and \eqref{eq:L.5}, we have $\ProbSpnDec[e^{-n_\theta}] = 0$ which implies the desired claim \eqref{eq:L.4}.
\end{proof}

\subsection{Proof of Proposition \ref{thm:M}} \label{sec:M}

Let $\big(W^{(0)}, \Spn, \mathcal N, (s_k, y_k,W^{(k)})_{k=1}^\infty\big)$
be the random elements constructed in Section \ref{sec:K}.
Our proof of Proposition \ref{thm:M} in the case  $\mathcal K> 0$ relies on the following lemma.

\begin{lemma} \label{thm:M.1}
	If $\mathcal K > 0$, then $Z^{(-\infty,0]}_0$ is an $\mathcal M$-valued random element.	
\end{lemma}
\begin{proof}
From \eqref{asp:H2}, we know that there exists a $t_0>0$ such that
\[
	\C\label{const:M.1}
	:=\sup\{|H_t f(x)|:t\geq t_0, f\in L_1^+(\nu), x\in E\} < \infty.
\]

\emph{Step 1.}
	We will show that $\mathrm Q$-almost surely
$
	Z^{(-t_0,0]}_0(\Ind_E)<\infty.
$
In fact, from Step 1 of the proof of Proposition \ref{thm:K}, we know
that, under $\mathrm Q$, $Z^{(-t_0,0]}_0(\Ind_E)$ and $Z^{(0,t_0]}_{t_0}(\Ind_E)$ have the same distribution.
From Lemma \ref{thm:K.6}, we know that they are both stochastically dominated by the random variable $X_{t_0}(\mathbf 1_E)$ under $\ProbQProc^{(\infty)}_\nu$.
Thus the desired result in this step is valid.

\emph{Step 2.}
We will show that $\mathrm Q$-almost surely
\[
\int_{(-\infty,t_0]\times \Omega} w_{-s}(\Ind_E) \mathcal N(\Diff s,\Diff w)<\infty.
\]
By \eqref{eq:K.44}, Campbell's theorem, \eqref{eq:K.41},\eqref{eq:L.3}, \eqref{asp:H2}, \eqref{asp:15} and \eqref{asp:H1} we get that
\begin{align}
	&\mathrm Q\Big[\int_{(-\infty,-t_0]\times \Omega} w_{-s}(\Ind_E) \mathcal N(\Diff s,\Diff w)\Big]
    = \mathrm Q\Big[\int_{-\infty}^{-t_0} 2\sigma(\Spn_s)^2\Diff s \int_\Omega w_{-s}(\Ind_E) \NMsr_{\Spn_s}(\Diff w)\Big]
      \\&= \mathrm Q\Big[\int_{-\infty}^{-t_0} 2\sigma(\Spn_s)^2(T_{-s}\Ind_E)(\Spn_s)\Diff s\Big]
      \leq 2 \|\sigma\|_\infty^2 \int_{-\infty}^{-t_0} \Diff s \int_E (T_{-s}\Ind_E)(x) \nu(\Diff x)
      \\&\leq 2 \|\sigma\|_\infty^2 \int_{-\infty}^{-t_0} \Diff s \int_E e^{-\lambda s}\phi(x)\nu(\Ind_E)\big(1+ (H_{-s}\Ind_E)(x)\big) \nu(\Diff x)
      \\&\leq 2 (1+\Cr{const:M.1}) \|\sigma\|_\infty^2 \int_{-\infty}^{-t_0} e^{-\lambda s}\Diff s
      <\infty.
\end{align}

\emph{Step 3.}
We will show that if $\mathcal K >0$ then $\mathrm Q$-almost surely
\[
	\sum_{k=1}^\infty
	W^{(k)}_{-s_k}(\Ind_E) \Ind_{(-\infty,-t_0]}(s_k)<\infty.
\]
To this end, we define the $\sigma$-algebra $\mathscr G := \sigma(\Spn, (s_k, y_k)_{k=1}^\infty)$. Then by
\eqref{eq:I.2}, \eqref{asp:H2}, \eqref{eq:K.43}
and Proposition
\ref{thm:L.2}, we have $\ProbSpnDec$-almost surely that
\begin{align}
	& \ProbSpnDec\Big(\sum_{k=1}^\infty W^{(k)}_{-s_k}(\Ind_E) \Ind_{(-\infty,-t_0]}(s_k)\Big|\mathscr G \Big)
      = \sum_{k=1}^\infty y_k \cdot (T_{-s_k} \Ind_E)(\Spn_{s_k}) \cdot \Ind_{(-\infty,-t_0]}(s_k)
       \\&= \sum_{k=1}^\infty y_k e^{-\lambda s_k} \phi(\Spn_{s_k}) \big(1+(H_{-s_k}\Ind_E)(x)\big)  \Ind_{(-\infty,-t_0]}(s_k)
       \\&\leq (1+\Cr{const:M.1}) \sum_{k=1}^\infty y_k e^{-\lambda s_k} \phi(\Spn_{s_k})  \Ind_{(-\infty,-t_0]}(s_k)
      < \infty.
\end{align}

\emph{Final Step.}
	From Steps $1$, $2$ and $3$, we know that $Z^{(-\infty,0]}_0(\Ind_E)<\infty$ almost surely provided $\mathcal K>0$.
Then one can use a routine
measure theoretic argument to get the desired result of this lemma.
\end{proof}

When $\mathcal K > 0$, the above lemma allows us to define a probability measure
$\widehat{\mathbf Q}$
on $\FntMsrSp$ as the distribution of
$Z^{(-\infty,0]}_0$
under the probability $\ProbSpnDec$.
It was mentioned in Section \ref{sec:E} that, after establishing Proposition \ref{thm:M}, one can also construct a probability measure $\mathbf Q_{\infty,\infty}$ using \eqref{eq:E.11} provided $\mathcal K>0$.
It will be explained later in Remark \ref{thm:M.4} that
$\widehat{\mathbf Q}$
and $\mathbf Q_{\infty,\infty}$ are exactly the same.

We are now ready to give the proof of Proposition \ref{thm:M}.

\begin{proof}[Proof of Proposition \ref{thm:M}]

\emph{Step 1.}
In this step, we will prove Proposition \ref{thm:M} in the case $\mathcal K > 0$.
	Let $F\in\Bndd C(\FntMsrSp)$ and $t\geq 0$ be arbitrary.
	Using Lemmas \ref{thm:Q.1}, \ref{thm:K.6} and Step 1 of the proof of Proposition \ref{thm:K}, we have
\begin{align}
&\ProbSupProc_\nu[\Ind_{\{X_t\neq \NullMsr\}} F(X_t) ]
=\ProbQProc^{(\infty)}_\nu[ e^{\lambda t} X_t(\phi)^{-1} F(X_t) ]
\\&=e^{\lambda t}\ProbSpnDec \Big[\frac{F(W^{(0)}_t + Z^{(0,t]}_t)}{W^{(0)}_t(\phi) +Z^{(0, t]}_t(\phi)}\Big]
=e^{\lambda t}\ProbSpnDec \Big[\frac{F(W^{(0)}_t+Z^{(-t,0]}_0)}{ W^{(0)}_t(\phi) +Z^{(-t, 0]}_0(\phi)}\Big]. \label{eq:M.3}
\end{align}
It follows from \eqref{eq:K.45} and Lemma \ref{thm:M.2} that
the $\FntMsrSp$-valued process $(W^{(0)}_t)_{t\geq 0}$ converges to $\NullMsr$ in probability when $t \uparrow \infty$.
It is also clear from \eqref{eq:K.51}, Lemma \ref{thm:M.1} and monotonicity that the following statement holds.
\begin{statement} \label{eq:M.35}
The $\FntMsrSp$-valued process
$(Z^{(-t,0]}_0)_{t>0}$
converges to $Z^{(-\infty,0]}_0$ almost surely as $t\uparrow \infty$.
\end{statement}
So by the continuous mapping theorem, we have
\begin{equation}
	\frac{F(W^{(0)}_t+Z^{(-t,0]}_0)}{ W^{(0)}_t(\phi) +Z^{(-t, 0]}_0(\phi)}
	\xrightarrow[t\to\infty]{}
Z_0^{(-\infty, 0]}(\phi)^{-1} F(Z^{(-\infty,0]}_0)
\end{equation}
in probability.
Notice also that the family of random variables
\begin{equation}
	\Big\{\Big|\frac{F(W^{(0)}_t+Z^{(-t,0]}_0)}{W^{(0)}_t(\phi) +Z^{(-t, 0]}_0(\phi)}\Big|:t\geq 1\Big\}
\end{equation}
is dominated by the random variable $Z^{(-1,0]}_0(\phi)^{-1}\cdot \sup_{\eta\in \FntMsrSp}|F(\eta)|$ which is $\ProbSpnDec$-integrable by \eqref{eq:K.8}.
Now we can apply the dominated convergence theorem in \eqref{eq:M.3} and get that
\begin{equation}\lim_{t\to\infty}e^{-\lambda t}\ProbSupProc_\nu[\Ind_{\{X_t\neq \NullMsr\}} F(X_t)   ]
=\ProbSpnDec[Z^{(-\infty, 0]}_0(\phi)^{-1}F(Z^{(-\infty,0]}_0)].
\end{equation}
Thus from Theorem \ref{thm:Y} and Proposition \ref{thm:K},
\begin{align}
&\int_{\FntMsrSp}F(\eta)\bfQ_{\infty,0}(\Diff\eta)
=\lim_{t\to\infty}\mathrm P_\nu [ F(X_t) | X_t\neq \NullMsr ]
\\&=\frac{\lim_{t\rightarrow\infty}e^{-\lambda t}\mathrm P_\nu[ \Ind_{\{X_t\neq \NullMsr\}} F(X_t) ]
}{\lim_{t\rightarrow\infty}e^{-\lambda t}\mathrm P_\nu(X_t\neq \NullMsr)}
=\mathcal K^{-1}\ProbSpnDec[Z^{(-\infty, 0]}_0(\phi)^{-1}F(Z^{(-\infty,0]}_0)]
\\\label{eq:M.4}&=\mathcal K^{-1}\int_{\FntMsrSp}\eta(\phi)^{-1}F(\eta)
\widehat{\bfQ}(\Diff\eta).
\end{align}
Since $F$ is arbitrary,
we can replace $F(\eta)$ in \eqref{eq:M.4} by $F(\eta)\cdot (\eta(\phi)\wedge n)$ where $n$ is an arbitrary
natural number.
Taking $n\uparrow \infty$ and using the monotone convergence theorem, we then arrive at
\begin{equation} \label{eq:M.41}
	\int_{\FntMsrSp}F(\eta) \cdot \eta(\phi)\bfQ_{\infty,0}(\Diff\eta)
	=\mathcal K^{-1}\int_{\FntMsrSp} F(\eta)
 \widehat{\bfQ}(\Diff\eta),
\quad F\in \BnddPos C(\FntMsrSp)
\end{equation}
which implies the desired result in this step.

\emph{Step 2.}
In this step, we will prove Proposition \ref{thm:M} in the case $\mathcal K = 0$.
According to Lemma \ref{thm:M.3}, there exists $t_0>0$ such that for any $t\geq t_0$ and $x\in E$, $2\nu(v_t)\phi(x) \geq v_t(x)$.
Now for any $t\geq t_0$,
we have from \eqref{eq:M.25} and \eqref{eq:Y.1} that
\begin{align}
	&\int_{\FntMsrSp}\eta(2\nu(v_t)\phi)\bfQ_{\infty,0}(\Diff\eta)
	\geq \int_{\FntMsrSp}\eta(v_t)\bfQ_{\infty,0}(\Diff\eta)
	\geq \int_{\FntMsrSp}(1 - e^{-\eta(v_t)})\bfQ_{\infty,0}(\Diff\eta)
      \\&= \int_{\FntMsrSp}\ProbSupProc_\eta(X_t \neq \NullMsr)\bfQ_{\infty,0}(\Diff\eta)
      = e^{\lambda t}.
\end{align}
From \eqref{eq:M.25}, Lemma \ref{thm:M.2} and Proposition \ref{thm:K} we have that
\begin{align}
	e^{-\lambda t}\nu(v_t)
	= e^{-\lambda t} \ProbSupProc_\nu(X_t\neq \NullMsr) \cdot \frac{-\log(1-\ProbSupProc_\nu(X_t\neq \NullMsr))}{\ProbSupProc_\nu(X_t\neq \NullMsr)}
	\xrightarrow[t\to\infty]{} \mathcal K \nu(\phi) = 0.
\end{align}
Now we have that
\[
	\int_{\FntMsrSp}\eta(\phi)\bfQ_{\infty,0}(\Diff\eta)
	\geq e^{\lambda t}/ (2\nu(v_t))
	\xrightarrow[t\to\infty]{} \infty
\]
which implies the desired result for this step.
\end{proof}

\begin{remark}
	\label{thm:M.4}
	From \eqref{eq:M.41} and \cite{MR1727226}*{Corollary 1.4}, it is clear that
 $\mathbf Q_{\infty,\infty} = \widehat{\mathbf Q}$ when $\mathcal K > 0$.
\end{remark}

\subsection{Proof of Proposition \ref{thm:T}} \label{sec:T}

Using the spine decomposition theorem (Lemma \ref{thm:K.6}), we can get the
following lemma.

\begin{lemma} \label{thm:T.1}
	If $\mathcal E<\infty$, then $\mathbf Q^\nu_{t,\infty}$ converges weakly to $\mathbf Q_{\infty,\infty}$ as $t\to\infty$.
\end{lemma}
\begin{proof}
	Using Lemmas \ref{thm:Q.1}, \ref{thm:K.6}, Step 1 of the proof of
Proposition \ref{thm:K},
Lemma \ref{thm:M.2}, \eqref{eq:M.35} and Remark \ref{thm:M.4}, we can verify that for any	$f\in \BnddPos C(E)$,
\begin{align}
&	 \int_\FntMsrSp e^{-\eta(f)} \mathbf Q^\nu_{t,\infty}(\Diff \eta)
= \ProbQProc^{(\infty)}_\nu[e^{-X_t(f)}]
= \mathrm Q[e^{-W_t^{(0)}(f)-Z_t^{(0,t]}(f)}]
\\&= \ProbSupProc_\nu[e^{-X_t(f)}]\mathrm Q[e^{-Z_0^{(-t,0]}(f)}]
\xrightarrow[t\to\infty]{} \mathrm Q[e^{-Z_0^{(-\infty,0]}(f)}]
= \int_\FntMsrSp e^{-\eta(f)} \mathbf Q_{\infty,\infty}(\Diff \eta).
\end{align}
Now the desired result follows from \cite{MR2760602}*{Theorem 1.18}.
\end{proof}
The above lemma  is a special case of Proposition \ref{thm:T}.
For the general case,
we will use the spine decomposition theorem
for superprocesses
with arbitrary initial value $\mu \in \FntMsrSp^o$.
Now the corresponding spine process will not be stationary, and can not be extended into a two-sided process in general.
Therefore, in this subsection, we construct the random elements
$(W^{(0)},\Spn, \mathcal N, (s_k,y_k,W^{(k)})_{k=1}^\infty)$ in a different probability space with respect to a new probability measure $\mathrm Q_\mu$, under which statements \eqref{eq:T.41}, \eqref{eq:K.42} (with $\RealLine$ replaced by $\RealLine_+$), \eqref{eq:K.43}, \eqref{eq:K.44} (with $\RealLine$ replaced by $\RealLine_+$) and \eqref{eq:K.45} (with $\mathrm P_\nu$ replaced by $\mathrm P_\mu$) hold.
\begin{statement} \label{eq:T.41}
$\Spn = (\Spn_t)_{t\geq 0}$
is an $E$-valued right continuous Markov process with transition semigroup $(S_t)_{t\geq 0}$ and initial distribution $\tilde \mu$.
\end{statement}
We present the spine decomposition for
superprocesses with arbitrary initial value in the following lemma.
We refer our readers to \cite{MR4058118} for its proof.
For any $0\leq a < b \leq t<\infty$,
and $f\in \BnddPos\Borel(E)$, let the random variable $Z_t^{(a,b]}(f)$ be defined as in \eqref{eq:K.51}.

\begin{lemma} \label{thm:T.5}
	Let $\mu\in \FntMsrSp^o$ be arbitrary.
	The $\FntMsrSp$-valued process
	$(W^{(0)}_t+ Z_t^{(0,t]})_{t\geq 0}$
under $\ProbSpnDec_\mu$ has
the same finite-dimensional distributions as the coordinate process
	$(X_t)_{t\geq 0}$
under $\ProbQProc^{(\infty)}_\mu.$
	Moreover, for any $t>0$, the $\FntMsrSp$-valued process $(Z_s^{(0,s]})_{s\in [0,t]}$
under $\ProbSpnDec_\mu$ has
	the same finite-dimensional distributions as the coordinate process $(X_s)_{s\in [0,t]}$
	under $\widetilde{\mu\NMsr}^{(t)}.$
\end{lemma}

 For any $t\geq 0, f\in\BnddPos\Borel(E)$ and $x\in E$, define
$
	\mathcal L_t f (x)
	:={\ProbSpnDec}_{\delta_x}[e^{-Z^{(0,t]}_{t}(f)}]$
and
$
      \mathcal L_\infty f (x) :=  \limsup_{t\to \infty}  \mathcal L_tf(x).
$

\begin{lemma} \label{thm:F}
	For any $x\in E$ and
$f\in \BnddPos\Borel(E)$,
$\lim_{t\to\infty}\mathcal L_tf(x) = \tilde \nu(\mathcal L_\infty f)$.
\end{lemma}

\begin{proof}
\emph{Step 1.}
We will show that for any $f\in \BnddPos\Borel(E)$ and $x\in E$, $\mathcal L_\infty f(x) \leq \tilde \nu(\mathcal L_\infty f)$.
To this end,
let us take arbitrary
$f\in \BnddPos\Borel(E)$,
$s,t\in \RealLine_+$ and $x\in E$, and
note that
\begin{equation}
 \mathcal L_{t+s}f(x)
 =\ProbSpnDec_{\delta_x}[e^{-Z^{(0,t+s]}_{t+s}(f)}]
 \leq \ProbSpnDec_{\delta_x}[e^{-Z^{(t,t+s]}_{t+s}(f)}]
  = \ProbSpnDec_{\delta_x}\big[\ProbSpnDec_{\delta_{\Spn_t}}[e^{-Z^{(0,s]}_{s}(f)}] \big]
 = S_t \mathcal L_s f (x).
\end{equation}
By \eqref{asp:H2},
\begin{equation}
	\mathcal L_{t+s}f(x)\leq S_t\mathcal L_{s}f(x)
= \phi(x)^{-1}e^{-\lambda t} T_t (\phi  \mathcal L_s f) (x)
= \big(1+ H_t(\phi \mathcal L_s f)(x)\big)\tilde \nu(\mathcal L_s f).	
\end{equation}
Noticing that, from \eqref{asp:H2}, $\lim_{t\to\infty}\sup_{x\in E, g\in \Bndd\Borel(E)}|H_tg(x)| = 0$.
 Therefore, letting $t\to\infty$,
 we have
 \begin{equation} \label{eq:F.1}
	\mathcal L_{\infty}f(x) \leq \tilde \nu(\mathcal L_s f). \end{equation}
	Now, taking $\limsup_{s\to\infty}$ in
\eqref{eq:F.1},
	using the reverse Fatou's lemma, we get the desired result for this step.

\emph{Step 2.}
We will show that for any $f\in \BnddPos\Borel(E)$, the Borel function $\mathcal L_tf$ on $E$ converges to the constant $\tilde \nu(\mathcal L_\infty f)$ in probability as $t\to \infty$
under $\tilde \nu$.
First note that, if $\tilde\nu(\mathcal L_\infty f) = 0$ then this is trivial from Step 1.
So let us assume that $\tilde\nu(\mathcal L_\infty f) > 0$ for the rest of this step.
Take arbitrary $\varepsilon_1, \varepsilon_2\in (0, 1)$, $t\geq 0$, and define
\begin{align}
&U_t := U_t^{\varepsilon_1}: =\{x\in E: \mathcal L_t f(x) >(1+\varepsilon_1) \tilde \nu(\mathcal L_\infty f)\};
\\&L_t:=L_t^{\varepsilon_2} : =\{x\in E: \mathcal L_t f(x) <(1-\varepsilon_2) \tilde \nu(\mathcal L_\infty f)\}.
\end{align}
From the reverse Fatou's lemma and Step 1, we have
\[
	\limsup_{t\to\infty}\tilde \nu(U_t) \leq \tilde \nu(\limsup_{t\to\infty} \Ind_{U_t}) = \tilde\nu(\cap_{t_0\geq 0}\cup_{t\geq t_0}U_t) = \tilde \nu(\emptyset) =0.
\]
Now we only need to show that  $\lim_{t\to\infty} \tilde\nu(L_t)=0.$
From \eqref{eq:F.1}, and the fact that the function $\mathcal L_tf$ is bounded by $1$, we have
\begin{align}
&\tilde\nu(\mathcal L_\infty f)
\leq \tilde \nu(\mathcal L_tf)
= \tilde \nu(\mathcal L_t f \cdot \Ind_{L_t}) + \tilde \nu(\mathcal L_t f \cdot \Ind_{U_t})+ \tilde \nu(\mathcal L_t f \cdot \Ind_{(L_t\cup U_t)^c})
\\& \leq (1-\varepsilon_2) \tilde \nu(\mathcal L_\infty f) \tilde \nu(L_t) + \tilde \nu(U_t) +(1+\varepsilon_1)\tilde \nu(\mathcal L_\infty f)\big(1-\tilde \nu(L_t)-\tilde \nu(U_t)\big)\\
  &\leq (1+\varepsilon_1)\tilde \nu(\mathcal L_\infty f)-(\varepsilon_1+\varepsilon_2) \tilde \nu(\mathcal L_\infty f)\tilde \nu(L_t)+ \tilde \nu(U_t).
\end{align}
Taking $\liminf_{t\to\infty}$,
we have
\begin{equation}
\tilde\nu(\mathcal L_\infty f)
\leq (1+\varepsilon_1)\tilde\nu(\mathcal L_\infty f)-(\varepsilon_1+\varepsilon_2)\tilde\nu(\mathcal L_\infty f)
\limsup_{t\to\infty} \tilde \nu(L_t^{\varepsilon_2}).
\end{equation}
Letting $\varepsilon_1\to 0$, we have
\begin{equation}
\tilde\nu(\mathcal L_\infty f)
\leq \tilde\nu(\mathcal L_\infty f)-\varepsilon_2\tilde\nu(\mathcal L_\infty f)\limsup_{t\to\infty}
\tilde \nu(L_t^{\varepsilon_2}).
\end{equation}
This is impossible unless
$\limsup_{t\to\infty} \tilde\nu(L^{\varepsilon_2}_t)=0$ holds as desired.

\emph{Step 3.}
We will show that for any $f\in \BnddPos\Borel(E)$ and $x\in E$, $\liminf_{t\to\infty}\mathcal L_t f(x) \geq \tilde \nu(\mathcal L_\infty f)$.
To this end, we fix arbitrary $f\in \BnddPos\Borel(E)$ and $x\in E$, and note that for any $t,s> 0$,
\begin{align}
	&\mathcal L_{t+s}f(x) \geq \ProbSpnDec_{\delta_x}\big[\exp\big\{-Z^{(t,t+s]}_{t+s}(f)\big\}; Z^{(0,t]}_{t+s} = \NullMsr\big]
      \\&= \ProbSpnDec_{\delta_x} \Big[\ProbSpnDec_{\delta_x}\big[\exp\big\{-Z^{(t,t+s]}_{t+s}(f)\big\}\big|\Spn_t\big] \cdot \ProbSpnDec_{\delta_x}\big(Z^{(0,t]}_{t+s} = \NullMsr\big |\Spn_t\big) \Big]
  \\&\label{eq:F.2}    = \ProbSpnDec_{\delta_x} \big[ \mathcal L_{s}f(\Spn_t)\cdot \ProbSpnDec_{\delta_x}\big(Z^{(0,t]}_{t+s} = \NullMsr\big |\Spn_t\big) \big].
\end{align}
On one hand, it can be verified that for any  $t>0$, the $\FntMsrSp$-valued process $(Z^{(0,t]}_{t+s})_{s\geq 0}$ under $\mathrm Q_{\delta_x}(\cdot|\tilde \xi_t)$ has the same finite-dimensional distributions as the coordinate process $(X_s)_{s\geq 0}$ under $\mathrm Q_{\delta_x}[\ProbSupProc_{Z^{(0,t]}_t}(\cdot)|\tilde \xi_t].$
In particular, due to Lemma \ref{thm:M.2} and the bounded convergence theorem, for each $t>0$, it holds that
$$
\ProbSpnDec_{\delta_x}\big(Z^{(0,t]}_{t+s} = \NullMsr\big |\Spn_t\big)
=\mathrm Q_{\delta_x}[\ProbSupProc_{Z^{(0,t]}_t}(X_s=\NullMsr)|\Spn_t]
\to 1\quad \mbox{as } s\to\infty.
$$
On the other hand, for any $\varepsilon > 0$,
defining
a function $F_{s,\varepsilon} \in \BnddPos\Borel(E)$ by
\[
F_{s,\varepsilon}(y):= \Ind_{\{|\mathcal L_sf(y) - \tilde \nu(\mathcal L_\infty f)|>\varepsilon\}}, \quad y\in E,
\]
we can verify from \eqref{eq:T.41}, \eqref{asp:H2} and Step 2 that, for any $t>0$,
\begin{align}
 &\ProbSpnDec_{\delta_x}\big(|\mathcal L_sf(\Spn_t)-\tilde \nu(\mathcal L_\infty f)|>\varepsilon\big)
 = S_t F_{s,\varepsilon}(x)
 = \phi(x)^{-1}e^{-\lambda t}T_t (\phi \cdot F_{s,\varepsilon} ) (x)
 \\&= \tilde \nu ( F_{s,\varepsilon} ) \big(1+(H_t F_{s,\varepsilon})(x) \big)
 \leq \tilde \nu ( F_{s,\varepsilon} ) \big(1+\sup_{g\in L_1^+(\nu),x\in E} |H_t g(x)| \big)
 \xrightarrow[s\to \infty]{} 0.
\end{align}
Therefore, for any $t>0$,
$\mathcal L_sf(\Spn_t)\cdot\ProbSpnDec_{\delta_x}\big(Z^{(0,t]}_{t+s} = \NullMsr\big |\Spn_t\big)$
converges to $\tilde \nu(\mathcal L_\infty f)$ in probability with respect to $\mathrm Q_{\delta_x}$ as $s\to\infty$.
By taking limit inferior in \eqref{eq:F.2} as $s\to \infty$ and using the
bounded convergence theorem, we get the desired result for this step.

\emph{Final Step.} Combine the results in Steps 1 and 3.
\end{proof}

We are now ready to give the proof of Proposition \ref{thm:T}.

\begin{proof}[Proof of Proposition \ref{thm:T}
in the case $\mathcal K > 0$]
	\emph{Step 1.}
	Let $f\in \BnddPos\Borel(E)$ and $t\geq 0$ be arbitrary.
	We will show that $\mathrm Q_\mu[e^{-Z_t^{(0,t]}(f)}] =  \tilde \mu(\mathcal L_tf).$
	To do this, we note that by Lemma \ref{thm:T.5},
	\[
		\mathcal L_tf(x)
		= \widetilde{\delta_x \NMsr}^{(t)}[e^{-X_t(f)}]
		= \NMsr_x\Big[e^{-X_t(f)} \frac{X_t(\phi)}{e^{\lambda t} \phi(x)}\Big], \quad x\in E.
	\]
	Therefore, again by Lemma \ref{thm:T.5} we have
	\begin{align}
		&\mathrm Q_\mu[e^{-Z_t^{(0,t]}(f)}]
		= \widetilde{\mu \NMsr}^{(t)}[e^{-X_t(f)}]
         = \int_E \NMsr_{x}\Big[e^{-X_t(f)}\frac{X_t(\phi)}{e^{\lambda t}\mu(\phi)}\Big]\mu(\Diff x)
	      \\&= \mu(\phi)^{-1}\int_E (\mathcal L_tf)(x)\cdot \phi(x)\mu(\Diff x)
	      =  \tilde \mu(\mathcal L_tf).
	\end{align}

\emph{Step 2.} We will show that for any
$\mu\in \FntMsrSp^o$
and $f\in \BnddPos\Borel(E)$,
\[
\lim_{t\to\infty}\int_\FntMsrSp e^{-\eta(f)}
\mathbf Q^\mu_{t,\infty}(\Diff \eta)
= \tilde \nu(\mathcal L_\infty f).
\]
In fact, from Lemmas \ref{thm:Q.1}, \ref{thm:T.5} and Step 1, we have that
\begin{equation}
	\int_\FntMsrSp e^{-\eta(f)}
\mathbf Q^\mu_{t,\infty}(\Diff \eta)
	= \ProbQProc^{(\infty)}_\mu[e^{-X_t(f)}]
	= \mathrm Q_\mu [e^{- W^{(0)}_t(f) - Z^{(0,t]}_t(f)}]
      = \ProbSupProc_\mu[e^{-X_t(f)}] \tilde{\mu}(\mathcal L_tf),
      \quad t\geq 0.
\end{equation}
Now the desired result in this step follows from Lemmas \ref{thm:M.2}, \ref{thm:F} and the bounded convergence theorem.

\emph{Final Step.}
From Lemma \ref{thm:T.1}  for any $f\in \BnddPos C(E)$,
\begin{align}
	\lim_{t\to\infty} \int_\FntMsrSp e^{-\eta(f)}
\mathbf Q^\nu_{t,\infty}(\Diff \eta)
= \int_\FntMsrSp e^{-\eta(f)} \mathbf Q_{\infty,\infty}(\Diff \eta).
\end{align}
Combining this with Step 2 and \cite{MR2760602}*{Theorem 1.18},
we get the desired result.
\end{proof}

\begin{proof}[Proof of Proposition \ref{thm:T}
in the case $\mathcal K = 0$]
We give a proof by contradiction. Assume that there exists $\mu\in \FntMsrSp^o$ such that $\mathbf Q^\mu_{t,\infty}$ converges weakly to a probability measure, say $\mathbf Q^*$ on $\mathcal M$.
	Then we have that
\[
	\int_\FntMsrSp e^{-\eta(\phi)}\mathbf Q^\mu_{t,\infty}(\Diff \eta)
	\geq \int_\FntMsrSp e^{-\eta(\|\phi\|_\infty \Ind_E)}\mathbf Q^\mu_{t,\infty}(\Diff \eta)
	\xrightarrow[t\to\infty]{} \int_\FntMsrSp e^{-\eta(\|\phi\|_\infty\Ind_E)}  \mathbf Q^*(\Diff \eta) > 0.
\]
On the other hand, from $xe^{-x} \leq 1$ for every $x\geq 0$, and Proposition \ref{thm:K}, we have
\[
	 \int_\FntMsrSp e^{-\eta(\phi)}\mathbf Q^\mu_{t,\infty}(\Diff \eta)
	\leq \int_\FntMsrSp \eta(\phi)^{-1}\mathbf Q^\mu_{t,\infty}(\Diff \eta)
	= e^{-\lambda t} \mu(\phi)^{-1}\ProbSupProc_\mu(X_t\neq \NullMsr)
	\xrightarrow[t\to\infty]{}
	0
\]
which is a contradiction.
\end{proof}

\subsection{Proof of Proposition \ref{thm:R}} \label{sec:R}

\begin{proof}[Proof of Proposition \ref{thm:R}
in the case $\mathcal K>0$]
	From \eqref{eq:Y.1} and \eqref{eq:Y.15} we have
	\begin{align}
		&\bfQ_{\infty,r}[F]
		=\int_{\FntMsrSp^o} e^{-\lambda r}\ProbSupProc_\eta[F(X_0)\Ind_{\{X_r\neq \NullMsr\}}] \bfQ_{\infty,0}(\Diff \eta)\\\label{eq:R.1}
        &=\int_{\FntMsrSp^o}
         F(\eta)
        e^{-\lambda r}\ProbSupProc_\eta(X_r\neq\NullMsr) \bfQ_{\infty,0}(\Diff \eta),
	\quad F\in \BnddPos\Borel(\FntMsrSp), r\geq 0.
\end{align}
Note that by Proposition \ref{thm:M},
$\eta(\phi)\bfQ_{\infty,0}(\Diff\eta)$ is a finite measure concentrated on $\FntMsrSp^o$, and that
by Lemma \ref{thm:Q.3} and Proposition \ref{thm:K},
for any $F\in\BnddPos\Borel(\FntMsrSp)$, and $r$ large enough,
\[
	\sup_{\eta\in\FntMsrSp^o}	
    F(\eta)
	e^{-\lambda r}\ProbSupProc_\eta(X_r\neq\NullMsr) \frac{1}{\eta(\phi)}
	\leq \|F\|_\infty e^{-\lambda r}\ProbSupProc_\nu(X_r\neq\NullMsr) \sup_{\eta\in \FntMsrSp^o} \frac{1}{\eta(\phi)} \frac{\ProbSupProc_\eta(X_r\neq\NullMsr)}{\ProbSupProc_\nu(X_r\neq\NullMsr)}
	<\infty.
\]
Now by Proposition \ref{thm:K} and the bounded convergence theorem we have
\begin{align}
		&\lim_{r\to\infty}\bfQ_{\infty,r}[F]
		=\int_{\FntMsrSp^o}
       F(\eta)
		\cdot \mathcal K \eta(\phi) \bfQ_{\infty,0}(\Diff \eta)
		= \bfQ_{\infty,\infty}[F],
		\quad F\in \BnddPos\Borel(\FntMsrSp)
\end{align}
	as desired.
\end{proof}

\begin{proof}[Proof of Proposition \ref{thm:R}
in the case $\mathcal K=0$]
We give a proof by contradiction.
Assume that $\bfQ_{\infty,r}$ converges
strongly to a probability measure, say $\bfQ^*$, as $r\to\infty$.
	Taking $F(\eta):= e^{-\eta(\phi)},\eta\in\FntMsrSp$ we have $\lim_{r\to\infty}\bfQ_{\infty,r}[F] = \bfQ^*[F]>0$.
	On the other hand,  we first observe from $\sup_{x\geq 0}xe^{-x} \leq 1$ that
    $F(\eta)\leq \eta(\phi)^{-1}$
	for every $\eta\in\FntMsrSp$.
	Then, noticing that \eqref{eq:R.1} still holds in this case,
	and also noticing from  Proposition \ref{thm:K} and Lemma \ref{thm:Q.3} that
\[
	\sup_{\eta\in\FntMsrSp^o }
	F(\eta)
	e^{-\lambda r}\ProbSupProc_\eta(X_r\neq\NullMsr)
	\leq e^{-\lambda r}\ProbSupProc_\nu(X_r\neq\NullMsr)\sup_{\eta\in\FntMsrSp^o }\frac{1}{\eta(\phi)}\frac{\ProbSupProc_\eta(X_r\neq\NullMsr)}{\ProbSupProc_\nu(X_r\neq\NullMsr)} \xrightarrow[r\to\infty]{}0.
\]
Using the bounded convergence theorem we have
$\lim_{r\to\infty}\bfQ_{\infty,r}[F] = 0$, which is a contradiction.
\end{proof}

\subsection{Proof of Proposition \ref{prop:new}}

\begin{lemma}\label{rate lemma}
                 If $\mathcal E<\infty$, then
		\[
		\lim_{t\to \infty}\sup_{r\geq 0}\dfrac{e^{\lambda r}\nu(v_t)}{\nu(v_{t+r})}=1.
		\]
\end{lemma}
\begin{proof}
	According to \cite{MR4175472}*{(3.10)},
	\begin{equation}
		\dfrac{e^{\lambda r}\nu(v_t)}{\nu(v_{t+r})}=\exp\left\{\int_t^{t+r}\dfrac{\nu(\Psi_0v_s)}{\nu(v_s)}\mathrm ds\right\}, \quad t > 0, r\geq 0,
	\end{equation}
	where
		$\Psi_0v_s(x):=\psi_0(x,v_s(x)), x\in E, s>0,$
	and
		\begin{equation}
			\psi_0(x,\lambda):=\sigma(x)^2\lambda^2+\int_0^\infty\left(e^{-\lambda u}-1+\lambda u\right)\pi(x,\mathrm du), \quad x\in E, \lambda \geq 0.
		\end{equation}
		To prove the desired result, it suffices to show
		\begin{equation} \label{eq:exponential}
			\int_t^{\infty}\dfrac{\nu(\Psi_0v_s)}{\nu(v_s)}\mathrm ds < \infty, \quad \text{~for~some~} t\geq 0.
		\end{equation}
	It is easy to see that for any $x\in E$,
	\[
	\dfrac{\partial \psi_0(x,\lambda)}{\partial \lambda}=2\sigma(x)^2\lambda+\int_0^\infty \left(1-e^{-\lambda u}\right)u\pi(x,\mathrm du)
	\]
	is a nonnegative increasing function with respect to $\lambda$ and $\psi_0(x,0)=0$.
	Thanks to the mean value theorem,
	\[
	\psi_0(x,\lambda)=\psi_0(x,\lambda)-\psi_0(x,0)\leq \lambda \dfrac{\partial \psi_0(x,\lambda)}{\partial \lambda}, \quad x\in E, \lambda \geq 0.
	\]
	Therefore,
	\[
	\Psi_0v_s(x)\leq v_s(x)\left.\dfrac{\partial \psi_0(x,\lambda)}{\partial \lambda}\right|_{\lambda=v_s(x)}, \quad x\in E, s> 0.
	\]
	By Lemma \ref{thm:M.3}, there exists $T_0>0$ such that $v_s(x)\leq 2\phi(x)\nu(v_s)$ for any $x\in E$ and $s>T_0$.  For $s>T_0$,
	\[
	\nu(\Psi_0v_s)\leq 2\nu(v_s)\left\langle\nu, \phi\left.\dfrac{\partial \psi_0(\cdot,\lambda)}{\partial \lambda}\right|_{\lambda=v_s(\cdot)}\right\rangle .
	\]
	Note that for any $x\in E$ and $s>T_0$,
	\begin{align}
		&\phi(x)\left.\dfrac{\partial \psi_0(\cdot,\lambda)}{\partial \lambda}\right|_{\lambda=v_s(x)}
		=2\sigma(x)^2\phi(x)v_s(x)+\int_0^\infty
		\left(1-e^{-v_s(x) u}\right) \phi(x) u\pi(x,\mathrm du)
		\\
		&\leq 2\| \sigma^2\phi \|_\infty v_s(x)+\int_0^\infty
		\left(1-e^{- 2\nu(v_s)\phi(x)u}\right)\phi(x)u\pi(x,\mathrm du).
	\end{align}
		Define a measure $\rho$ on $(0,\infty)$ so that for any non-negative Borel function $f$ on $(0,\infty)$,
		\[
		\int_{(0,\infty)} f(u) \rho(\mathrm du) = \int_E \nu(\mathrm dx) \int_{(0,\infty)} f(\phi(x)u) \pi(x,\mathrm du).
		\]
	Then for any $s> T_0$,
	\[
	\nu(\Psi_0v_s)\leq
	2\nu(v_s)\left[2\| \sigma^2\phi \|_\infty \nu(v_s)+ \int_0^\infty \left(1-e^{- 2\nu(v_s)u}\right)u\rho(\mathrm du)
	\right].
	\]
		Now the integral on the left hand side of \eqref{eq:exponential} can be bounded by
		\[
		\int_t^{\infty}\dfrac{\nu(\Psi_0v_s)}{\nu(v_s)}\mathrm ds
		\leq 4\| \sigma^2\phi \|_\infty \textrm{I}_{t}+2\textrm{II}_{t}, \quad t>T_0, r>0,
		\]
		where
		\[
		\textrm{I}_{t} := \int_t^{\infty} \nu(v_s) \mathrm ds, \text{~and~} \textrm{II}_t := \int_t^\infty \mathrm ds \int_0^\infty \left(1-e^{- 2\nu(v_s)u}\right)u\rho(\mathrm du).
		\]
	From \eqref{eq:M.25} and Propositions \ref{thm:K} and \ref{thm:L}, we know that $e^{-\lambda t} \nu(v_t) \to \mathcal K>0$ as $t\to \infty$.
		In particular,
		there exist $T_1\geq T_0$ and $\C\label{c:temp}>0$ such that $\nu(v_s)\leq \Cr{c:temp}e^{\lambda s}$ for every $s\geq T_1$.
		This yields $\textrm{I}_{t}<\infty$ for $t\geq T_1$.
	Note that
	$\int_0^\infty \left(1-e^{- 2\theta u}\right)u\rho(\mathrm du)$ is an increasing function with respect to $\theta$.
		Thus for sufficiently large $t$ so that $ t \geq T_1$ and $2\Cr{c:temp}e^{\lambda t}\leq 1$,
		\begin{align}
			&\textrm{II}_{t}
			\leq   \int_t^\infty \mathrm ds\int_0^\infty \left(1-e^{- 2\Cr{c:temp}e^{\lambda s}u}\right)u\rho(\mathrm du)
			=   \frac{1}{|\lambda|}\int_0^\infty u\rho(\mathrm du)  \int_{0}^{2\Cr{c:temp}ue^{\lambda t}} \frac{1-e^{- z}}{z}\mathrm d z
			\\&\leq  \dfrac{1}{|\lambda|}\int_0^\infty u\rho(\mathrm du)\int_0^{u} \dfrac{1-e^{-z}}{z}\mathrm ds
			\leq \dfrac{1}{|\lambda|}\int_0^1 u^2\rho(\mathrm du)+\dfrac{1}{|\lambda|}\int_1^\infty u(1+\log u)\rho(\mathrm du).
		\end{align}

		When $\mathcal E<\infty$, it is easy to see that $\textrm{II}_{t}<\infty$.
	The proof is complete.
\end{proof}
	\begin{lemma}\label{lem:bounded}
		Suppose $\mathcal E<\infty$.  Then for any $\mu\in\FntMsrSp^o$,
		\[
		\lim_{t\to \infty}\sup_{r\geq 0}\sup_{\eta\in\FntMsrSp^o}\dfrac{\mu(\phi)e^{\lambda t}\ProbSupProc_{\eta}\left(X_r\neq 0\right)}{\eta(\phi)\ProbSupProc_{\mu}\left(X_{t+r}\neq 0\right)}\leq 16.
		\]
	\end{lemma}

	\begin{proof}
		It is well known that $1-e^{-u}\leq u$ and that there is a $\delta>0$ such that  $1-e^{-u}\geq u/2$ for $u\in [0, \delta]$.
                 By \cite{MR4175472}*{(3.39)} we have $\lim_{t\to\infty}\nu(v_t)=0$.
                 By Lemma \ref{thm:M.3}, given $\mu\in\mathcal M^o$, there exists $T_2(\mu)>0$ such that $\mu(v_t)\leq \delta$ for $t\geq T_2(\mu)$.
                 Therefore, for any $\eta, \mu \in \mathcal M^o$, $t\geq T_2(\mu)$ and $r\geq 0$, we have
		\[
		\dfrac{\ProbSupProc_{\eta}\left(X_r\neq 0\right)}{\ProbSupProc_{\mu}\left(X_{t+r}\neq 0\right)}
		=\dfrac{1-e^{-\eta(v_r)}}{1-e^{-\mu(v_{t+r})}}
		\leq \dfrac{2\eta(v_r)}{\mu(v_{t+r})}.
		\]
		The uniform lower and the upper bounds of $v_r(x)$ are given in Lemma \ref{thm:M.3} as well:
		There is a $T_3\geq 0$, such that
                 $\phi(x)\nu(v_t)/2\leq v_t(x) \leq 2\phi(x)\nu(v_t) $ for $t\geq T_3$ and $x\in E$.
		From Lemma \ref{rate lemma}, there exists a $T_4>0$, such that $e^{\lambda t} \nu(v_r)/ \nu(v_{t+r})\leq 2$ for every $t\geq T_4$ and $r>0$.
		Now for any $\eta,\mu \in \mathcal M^o$, $t\geq T_2(\mu)\vee T_3\vee T_4$ and $r>0$,
		\begin{equation}
			\dfrac{\mu(\phi)e^{\lambda t}\ProbSupProc_{\eta}\left(X_r\neq 0\right)}{\eta(\phi)\ProbSupProc_{\mu}\left(X_{t+r}\neq 0\right)}\leq
			\dfrac{\mu(\phi)e^{\lambda t}}{\eta(\phi)}\dfrac{4\eta(\phi)\nu(v_r)}{\mu(\phi)\nu(v_{t+r})/2}
			\leq \dfrac{8e^{\lambda t}\nu(v_r)}{\nu(v_{t+r})} \leq 16.
		\end{equation}
		The proof is complete.
	\end{proof}

\begin{proof}[Proof of Proposition \ref{prop:new}]
		Fix arbitrary $\mu \in \mathcal M^o$ and non-negative continuous function $f$ on $E$.
	For any $t,r>0$,
	\begin{align}
		&\ProbSupProc_{\mu}  \left( e^{-X_t(f)} \middle|X_{t+r}\neq 0\right)
		=\dfrac{\ProbSupProc_{\mu}\left(e^{-X_t(f)} \mathbf 1_{\{X_{t+r}\neq 0\}}\right)}{\ProbSupProc_{\mu}(X_{t+r}\neq 0)}
		\\&=\dfrac{\ProbSupProc_{\mu}\left(e^{-X_t(f)}\ProbSupProc_{X_t}( X_r \neq 0)\right)}{\ProbSupProc_{\mu}(X_{t+r}\neq 0)}
		=\widetilde{\ProbSupProc}^{(\infty)}_{\mu}\left(e^{-X_t(f)}\dfrac{\mu(\phi)e^{\lambda t}\ProbSupProc_{X_t}(X_{r}\neq 0)}{X_t(\phi)\ProbSupProc_{\mu}(X_{t+r}\neq 0)}\right)
	\end{align}
	where probability $\widetilde{\ProbSupProc}^{(\infty)}_{\mu}$ is given as in Lemma \ref{thm:Q.1}.
		According to the Skorohod representation theorem (see \cite{MR4226142}*{Theorem 5.31}, for example)
	there exists an $\mathcal M$-valued process $(\hat X_t)_{t\geq 0}$ converging almost surely to an $\mathcal M$-valued random element $\hat X_\infty$ on some probability space $(\hat\Omega, \hat {\mathcal F}, \hat{\ProbSupProc})$ so that the law of $\hat X_\infty$ is $\mathbf Q_{\infty,\infty}$ and the law of $\hat X_t$ is $\mathbf Q_{t,\infty}^\mu$ for every $t\geq 0$.
		Now we have
		\begin{equation}
			\ProbSupProc_{\mu}  \left( e^{-X_t(f)} \middle|X_{t+r}\neq 0\right)
			=\hat{\ProbSupProc}\left(e^{-\hat X_t(f)}\dfrac{\mu(\phi)e^{\lambda t}\ProbSupProc_{\hat X_t}(X_{r}\neq 0)}{\hat X_t(\phi)\ProbSupProc_{\mu}(X_{t+r}\neq 0)}\right), \quad t,r\geq 0.
		\end{equation}
                 Note that, by Lemma \ref{lem:bounded}, there exists $T_5(\mu)>0$ such that for any $t\geq T_5(\mu)$ and $r\geq 0$,
		\begin{equation}
			e^{-\hat X_t(f)}\dfrac{\mu(\phi)e^{\lambda t}\ProbSupProc_{\hat X_t}\left(X_r\neq 0\right)}{\hat X_t(\phi)\ProbSupProc_{\mu}\left(X_{t+r}\neq 0\right)}\leq 17.
		\end{equation}
		Also notice that $e^{-\hat X_t(f)}$ converges almost surely to $e^{-\hat X_\infty (f)}$ as $t\to \infty$.
		So now if one can show that almost surely
		\begin{equation}\label{eq:claim}
			\lim_{t,r\to \infty}\dfrac{\mu(\phi)e^{\lambda t}\ProbSupProc_{\hat X_t}\left(X_r\neq 0\right)}{\hat X_t(\phi)\ProbSupProc_{\mu}\left(X_{t+r}\neq 0\right)} = 1,
		\end{equation}
		then by the bounded convergence theorem we get the desire result for this proposition.
		
		Let us now verify \eqref{eq:claim}.
		From \eqref{eq:M.25}, we have
		\begin{equation}\label{eq:fraction}
			\dfrac{\mu(\phi)e^{\lambda t}\ProbSupProc_{\hat X_t} \left(X_r\neq 0\right)}{\hat X_t(\phi)\ProbSupProc_{\mu}\left(X_{t+r}\neq 0\right)}
			=\dfrac{1-e^{-\hat X_t(v_r)}}{\hat X_t(v_r)}\dfrac{\mu(v_{t+r})}{1-e^{-\mu(v_{t+r})}}\dfrac{\mu(\phi)e^{\lambda t}\hat X_t(v_r)}{\hat X_t(\phi)\mu(v_{t+r})}, \quad t,r\geq 0.
		\end{equation}
By Lemma \ref{thm:M.3}, we  have for $t,r\geq 0$,
		\[
		\nu(v_r) \hat X_t (\phi)\Big(1-\sup_{x\in E}|\Cr{const:M.3}(r,x)|\Big)
		\leq \hat X_t(v_r)
		\leq \nu(v_r)\hat X_t(\phi)\Big(1+\sup_{x\in E}|\Cr{const:M.3}(r,x)|\Big),
		\]
		with $\lim_{r\to\infty}\sup_{x\in E}|\Cr{const:M.3}(r,x)| = 0$.
                By \cite{MR4175472}*{(3.39)} we have $\lim_{r\to\infty}\nu(v_r)=0$ and that
		\begin{equation}
			\limsup_{t\to \infty} \hat X_t(\phi)\leq \lim_{t\to \infty} \hat X_t(\|\phi\|_\infty \mathbf 1_E) = \hat X_\infty(\|\phi\|_\infty \mathbf 1_E)< \infty, \quad \text{a.s.}
		\end{equation}
                 Thus $\lim_{t,r\to \infty}\hat X_t(v_r)  = 0$ almost surely. Therefore,
		\[
\lim_{t,r\to\infty}\dfrac{1-e^{-\hat X_t(v_r)}}{\hat X_t(v_r)}=1, \quad \text{a.s.}
		\]
		Similarly we have for every $t,r\geq 0$,
		\[
		\nu(v_{t+r}) \mu (\phi)\Big(1-\sup_{x\in E}|\Cr{const:M.3}({t+r},x)|\Big)
		\leq \mu(v_{t+r})
		\leq \nu(v_{t+r})\mu(\phi)\Big(1+\sup_{x\in E}|\Cr{const:M.3}(t+r,x)|\Big),
		\]
		and
		\[\lim_{t,r\to\infty}\dfrac{\mu(v_{t+r})}{1-e^{-\mu(v_{t+r})}}=1.\]
         Using Lemma \ref{rate lemma} for the third fraction
		on the right hand side of \eqref{eq:fraction},
         we get
		\begin{align}
			&\limsup_{t,r\to\infty}\dfrac{\mu(\phi)e^{\lambda t}\hat X_t(v_r)}{\hat X_t(\phi)\mu(v_{t+r})}
			\leq  \limsup_{t,r\to\infty}\dfrac{\mu(\phi)e^{\lambda t}\nu(v_r)\hat X_t(\phi)(1+\sup_{x\in E}|\Cr{const:M.3}(r,x)|)}{\hat X_t(\phi)\nu(v_{t+r})\mu(\phi)(1-\sup_{x\in E}|\Cr{const:M.3}(t+r,x)|)}
			\\&=\limsup_{t,r\to\infty}\dfrac{e^{\lambda t}\nu(v_r)}{\nu(v_{t+r})}=1.
		\end{align}
	On the other hand,
		\begin{align}
			&\liminf_{t,r\to\infty}\dfrac{\mu(\phi)e^{\lambda t}\hat X_t(v_r)}{\hat X_t(\phi)\mu(v_{t+r})}
			\geq  \liminf_{t,r\to\infty}\dfrac{\mu(\phi)e^{\lambda t}\nu(v_r)\hat X_t(\phi)
				(1-\sup_{x\in E}|\Cr{const:M.3}(r,x)|)}{\hat X_t(\phi)\nu(v_{t+r})\mu(\phi)(1+\sup_{x\in E}|\Cr{const:M.3}(t+r,x)|)}
			\\&=\liminf_{t,r\to\infty}\dfrac{e^{\lambda t}\nu(v_r)}{\nu(v_{t+r})}=1.
		\end{align}
	The proof is complete.
\end{proof}

\appendix

\section{} \label{app:A}

\begin{lemma} \label{thm:A.05}
	The transition semigroup $(Q_t)_{t\geq 0}$ given
	in \eqref{eq:I.13} preserves $\Bndd\Borel(\FntMsrSp)$.
\end{lemma}
\begin{proof}
	Denote $\overline E:= E\cup \{\partial\}$, where $\partial$ is an isolated point not contained in $E$.
	The Polish space of all finite Borel measures on $\overline E$, equipped with the topology of weak convergence, is denoted by $\FntMsrSp(\overline E)$.
	Define a
	conservative Borel right transition semigroup $(\overline P_t)_{t\geq 0}$ on $(\overline E,\Borel (\overline E))$ using \cite{MR2760602}*{(A.20)}.
	Let $\overline \xi$ be a Borel right process with transition semigroup $(\overline P_t)_{t\geq 0}$.
	Define $\overline \psi$ as the extension of $\psi$ on $\overline E \times \RealLine_+$ such that $\overline \psi(\partial,\cdot) \equiv 0$.
	Let $\overline X$ be a $(\overline \xi, \overline \psi)$-superprocess whose transition semigroup is denoted by $(\overline Q_t)_{t\geq 0}$.
According to \cite{MR2760602}*{Theorem 5.11}, $(\overline Q_t)_{t\geq 0}$ is a Borel right transition semigroup on $\mathcal M(\overline E)$.
	Define a map $\Gamma: \mathcal M(\overline E) \to \FntMsrSp$ so that for any $\mu\in \mathcal M(\overline E)$, the measure $\Gamma \mu$ is the restriction of the set function $\mu$ on $\Borel (E)$.
	Define a map $\Lambda: \FntMsrSp \to \mathcal M(\bar E)$ so that for any $\mu \in \FntMsrSp$, the measure $\Lambda \mu$ on $\overline E$ is the unique extension of $\mu$ on $\Borel (\overline E)$ so that $\Lambda \mu (\{\partial\}) = 0$.
	Obviously we have $\Gamma \circ \Lambda$ is the identical map on $\FntMsrSp$; and from the fact that $\partial$ is an isolated point in $\overline E$ we know $\Gamma$ and $\Lambda$ are continuous maps.
	Fix an arbitrary $t\geq 0$ and $F\in \Bndd\Borel(\FntMsrSp)$.
	It can be verified (see the proof of \cite{MR2760602}*{Theorem 5.12}) that
	$\overline Q_t(F\circ \Gamma) (\overline \mu) = (Q_t F) \circ \Gamma(\overline \mu)$ for each $\overline \mu \in \mathcal M(\overline E)$.
	From this we can verify that $Q_t F= (\overline Q_t(F\circ \Gamma) ) \circ \Lambda$ is a real valued bounded Borel function on $\FntMsrSp$.
	Therefore, the semigroup $(Q_t)_{t\geq 0}$ preserves $\Bndd\Borel(\FntMsrSp)$ as desired.
\end{proof}

\begin{proof}[Proof of Lemma \ref{thm:L.2} (1)]
Fix arbitrary $0<\delta < \epsilon$.
From \eqref{eq:K.42}, we have
\begin{align}
	\sum_{k=1}^\infty \Ind_{(-\infty,0]}(s_k) y_k e^{\epsilon s_k} \phi(\Spn_{s_k})
	=\int_{\RealLine\times \RealLine_+} \Ind_{\{s\leq 0\}} y e^{\epsilon s}\phi(\Spn_s) \mathcal D(\Diff s,\Diff y)
	= \text{I}+\text{II},
\end{align}
	where
\begin{align}
	&\text{I}
	:=\int_{\RealLine\times \RealLine_+} \Ind_{\{y\leq \frac{e^{-\delta s}}{\phi(\Spn_{s})},s\leq 0\}}y e^{\epsilon s}\phi(\Spn_{s}) \mathcal D(\Diff s,\Diff y),
	\\ & \text{II}
	:= \int_{\RealLine\times \RealLine_+} \Ind_{\{y> \frac{e^{-\delta s}}{\phi(\Spn_{s})},s\leq 0\}} ye^{\epsilon s}\phi(\Spn_{s}) \mathcal D(\Diff s,\Diff y).
\end{align}
We first show that $\text{II}< \infty$, $\ProbSpnDec$-a.s.
In fact, from \eqref{eq:K.41} and \eqref{eq:I.35},
\begin{align}
	&\ProbSpnDec\Big[\int_{-\infty}^0  \Diff s \int_{\frac{e^{-\delta s}}{\phi(\Spn_{s})}}^\infty y \pi(\Spn_s, \Diff y) \Big]
	=\int_{-\infty}^0 \Diff s \int_E
\tilde \nu(\Diff x)
\int_{\frac{e^{-\delta s}}{\phi(x)}}^\infty y\pi(x, \Diff y)
	\\&=\int_{-\infty}^0 \Diff s \int_E \phi(x) \nu(\Diff x)\int_{\frac{e^{-\delta s}}{\phi(x)}}^\infty y\pi(x, \Diff y)
	\\&= \int_E\nu(\Diff x)\int_{\frac{1}{\phi(x)}}^\infty y\phi(x) \pi(x, \Diff y)
\int_{-\frac{\log (y\phi(x))}{\delta}}^{0}\Diff s
	=\frac{\mathcal E}{\delta}
	<\infty.
\end{align}
	Therefore, we have $\ProbSpnDec$-almost surely
\begin{align}
 &\int_{\RealLine\times \RealLine_+}\Big[1\wedge \Big(\Ind_{\{y> \frac{e^{-\delta s}}{\phi(\Spn_{s})},s\leq 0\}} ye^{\epsilon s}\phi(\Spn_{s})\Big) \Big]  \Diff s \cdot y \pi(\Spn_s, \Diff y)
	\\&\leq   \int_{\RealLine\times \RealLine_+}  \Ind_{\{y> \frac{e^{-\delta s}}{\phi(\Spn_{s})},s\leq 0\}} \Diff s\cdot y \pi(\Spn_s, \Diff y)
	< \infty.
\end{align}
Now from \eqref{eq:K.42} and \cite{MR3155252}*{Theorem 2.7(i)} we have $\ProbSpnDec(\text{II} < \infty | \Spn) = 1$.
We then show that $\text{I} < \infty$, $\ProbSpnDec$-a.s.
	In fact, from \eqref{eq:K.41} and \eqref{eq:K.42},
\begin{align}
	&\ProbSpnDec[\text{I}]
	= \ProbSpnDec\big[\ProbSpnDec[\text{I}|\xi]\big]
	=\ProbSpnDec\Big[\int_{-\infty}^0 \Diff s \int_0^{\frac{e^{-\delta s}}{\phi(\Spn_s)}} ye^{\epsilon s}\phi(\Spn_s) y \pi(\Spn_s, \Diff y)\Big]
	\\&=\ProbSpnDec \Big[ \int_{-\infty}^0 \Diff s \int_0^{1\wedge \frac{e^{-\delta s}}{\phi(\Spn_s)}} e^{\epsilon s} y^2 \phi(\Spn_s)\pi(\Spn_s, \Diff y) + \int_{-\infty}^0 \Diff s \int_{1\wedge \frac{e^{-\delta s}}{\phi(\Spn_s)}}^{\frac{e^{-\delta s}}{\phi(\Spn_s)}} e^{\epsilon s}y\phi(\Spn_s) y \pi(\Spn_s, \Diff y)\Big]
	\\&\leq \ProbSpnDec \Big[ \int_{-\infty}^0 \Diff s \int_0^1 e^{\epsilon s} y^2 \phi(\Spn_s)\pi(\Spn_s, \Diff y) + \int_{-\infty}^0 \Diff s \int_{1}^{\infty} e^{(\epsilon-\delta) s} y \pi(\Spn_s, \Diff y)\Big]
     \\& \leq \|\phi\|_\infty\int_{-\infty}^0  e^{\epsilon s}\Diff s \int_E \tilde \nu(\Diff x)\int_0^1y^2 \pi(x, \Diff y)
 + \int_{-\infty}^0  e^{(\epsilon-\delta )s}\Diff s\int_E \tilde \nu(\Diff x)\int_1^{\infty}y \pi(x, \Diff y)
	\\&<\infty.
\end{align}
Now the desired result of this lemma follows.
\end{proof}
\begin{proof}[Proof of Lemma \ref{thm:L.2} (2)]
	Fix the arbitrary $\epsilon >0$ and $s_0\geq 0$.
	Define a constant
\begin{equation} \label{eq:E.934}
	K
	:= \max \{\|\phi\|_\infty,e^{\epsilon s_0}\},
\end{equation}
	and random variables
\begin{equation}
	\eta_T
	:= \int_{-T}^{0} \Diff s \int_{\frac{Ke^{-\epsilon s}}{\phi(\Spn_s)}}^\infty y\pi(\Spn_s,\Diff y),
	\quad T \in (0,\infty].
\end{equation}

	\emph{Step 1.}
	We will show that $\ProbSpnDec[\eta_\infty]=\infty.$
In particular, this implies that
there exists a $t_1>0$ such that $\ProbSpnDec[\eta_T]>0$ for all $T\geq t_1$.
To show that $\ProbSpnDec[\eta_\infty]=\infty,$
we note from  \eqref{eq:K.41} and Fubini's theorem  that
\begin{align}
	&\ProbSpnDec[\eta_\infty]
     = \int_{-\infty}^{0} \Diff s \int_E \tilde\nu(\Diff x) \int_{\frac{Ke^{-\epsilon s}}{\phi(x)}}^\infty y\pi(x, \Diff y)
	\\ &= \int_E\phi(x)\nu(\Diff x)\int_{\frac{K}{\phi(x)}}^\infty y\pi(x, \Diff y)
\int_{-\frac{1}{\epsilon}\log(\frac{y\phi(x)}{K})}^0\Diff s
	\\&=\frac{1}{\epsilon}\int_E\phi(x)\nu(\Diff x)\int_{\frac{K}{\phi(x)}}^\infty
\Big(\log\big(y\phi(x)\big)-\log K\Big)y\pi(x, \Diff y)
	\\\label{eq:E.935}&\geq \frac{1}{\epsilon}\int_E\nu(\Diff x) \int_{\frac{K}{\phi(x)}}^\infty y\phi(x)
\log \big(y\phi(x)\big)\pi(x, \Diff y)-\frac{A}{\epsilon},
\end{align}
	where
\begin{equation} \label{eq:E.936}
	A
	:= \ln K \cdot \sup_{x\in E} \int_1^\infty y \pi(x,\Diff y)
	< \infty.
\end{equation}
	Since
\[
	\int_E\nu(\Diff x) \int_{\frac{1}{\phi(x)}}^\infty y\phi(x)
\log \big(y\phi(x)\big)\pi(x, \Diff y)
	= \mathcal E
	=\infty
\]
	and
\begin{align}
	\int_E\nu(\Diff x) \int_{\frac{1}{\phi(x)}}^{\frac{K}{\phi(x)}} y\phi(x)
\log \big(y\phi(x)\big)\pi(x, \Diff y)
\leq K \log K \int_E \nu(\Diff x) \int_{\frac{1}{\|\phi\|_\infty}}^\infty \pi(x, \Diff y)
	<\infty,
\end{align}
	we get that
\begin{equation}
	\int_E\nu(\Diff x) \int_{\frac{K}{\phi(x)}}^\infty y\phi(x)
\log \big(y\phi(x)\big)\pi(x, \Diff y)
	= \infty.
\end{equation}
	Now the desired result in this step follows from \eqref{eq:E.935}, \eqref{eq:E.936} and above.

\emph{Step 2.}
	We will show that $\ProbSpnDec[\eta_T]<\infty$ for all $T\in (0,\infty)$.
	From \eqref{eq:K.41}, \eqref{eq:E.934} and Fubini's theorem, we have
\begin{align}
	\label{eq:E.9375}&\ProbSpnDec[\eta_T]
     = \int_{-T}^{0} \Diff s \int_E \tilde \nu(\Diff x)
 \int_{\frac{Ke^{-\epsilon s}}{\phi(x)}}^\infty y\pi(x, \Diff y)
	\leq \int_{-T}^{0} \Diff s \int_E \tilde \nu(\Diff x) \int_1^\infty y\pi(x, \Diff y)
	\\& \leq T \cdot \sup_{x\in E} \int_1^\infty y\pi(x, \Diff y)
	< \infty.
\end{align}
	
\emph{Step 3.}
	We will show that there exists a $t_2>0$ such that for any $t>t_2$, $x\in E$, and  $f\in \BnddPos\Borel(E)$, it holds that
$S_tf(x) \leq 2\tilde\nu( f).$
	In fact, let $H$ be as in \eqref{asp:H2}, then there exists a $t_2>0$ such that for any $t>t_2$, $x\in E$ and $f\in \BnddPos\Borel(E)$,
	$|H_{t}(\phi f)(x)|\leq 1$.
	Now for any $t>t_2$, $x\in E$ and $f\in \BnddPos\Borel(E)$, we can verify from \eqref{eq:E.17} and \eqref{asp:H2} that
\begin{align}
	&S_tf(x)
	= \frac{1}{\phi(x)e^{\lambda t}}T_t (\phi f) (x)
	=  \nu(\phi f) (1+ H_{t}(\phi f)(x))
    \leq 2\tilde \nu(f).
\end{align}

\emph{Step 4.}
Let $t_0:= \max\{t_1,t_2\}$.
We will show that there exists a constant $\C\label{const:E.9376}>0$ such that for all $T> t_0$, it holds that $\ProbSpnDec[\eta_T^2]\leq \Cr{const:E.9376}\ProbSpnDec[\eta_T]^2$.
	To do this we note that for any $T>t_0$,
\begin{align}
	&\ProbSpnDec[\eta_T^2]
	=\ProbSpnDec \Big[\Big(\int_{-T}^0\Diff t\int_{\frac{Ke^{-\epsilon t}}{\phi(\Spn_t)}}^\infty r \pi(\Spn_t, \Diff r) \Big)\cdot \Big(\int_{-T}^0\Diff s\int_{\frac{Ke^{-\epsilon s}}{\phi(\Spn_s)}}^\infty u \pi(\Spn_s, \Diff u)\Big) \Big]
	\\&=2\ProbSpnDec\Big[ \int_{-T}^0\Diff t\int_{\frac{K e^{-\epsilon t}}{\phi(\Spn_t)}}^\infty r \pi(\Spn_t, \Diff r)\int_t^0\Diff s
	\int_{\frac{Ke^{-\epsilon s}}{\phi(\Spn_s)}}^\infty u \pi(\Spn_s, \Diff u) \Big]
	=\text{III}+\text{IV}
\end{align}
	where
\begin{align}
	\text{III}
	:= 2\ProbSpnDec\Big[\int_{-T}^0\Diff t\int_{\frac{Ke^{-\epsilon t}}{\phi(\Spn_t)}}^\infty r \pi(\Spn_t, \Diff r)\int_t^{(t+t_0)\wedge 0}\Diff s\int_{\frac{Ke^{-\epsilon s}}{\phi(\Spn_s)}}^\infty u \pi(\Spn_s,\Diff u) \Big]
\end{align}
	and
\begin{align}
	\text{IV}
	:= 2\ProbSpnDec\Big[\int_{-T}^0\Diff t \int_{\frac{Ke^{-\epsilon t}}{\phi(\Spn_t)}}^\infty r \pi(\Spn_t, \Diff r) \int_{(t+t_0)\wedge 0}^0\Diff s\int_{\frac{Ke^{-\epsilon s}}{\phi(\Spn_s)}}^\infty u \pi(\Spn_s, \Diff u) \Big].
\end{align}

We first show that there exists a constant $\C\label{const:E.94}>0$ such that for all $T>t_0$, it holds that $ \text{III} \leq \Cr{const:E.94} \ProbSpnDec[\eta_T]^2.$
	In fact, note $t\mapsto \ProbSpnDec[\eta_t]$ is non-decreasing, we have for any $T > t_0$,
\begin{align}
	&\text{III}
	\leq 2\ProbSpnDec\Big[\int_{-T}^0\Diff t\int_{\frac{Ke^{-\epsilon t}}{\phi(\Spn_t)}}^\infty r \pi(\Spn_t, \Diff r)\int_t^{t+t_0}\Diff s\int_{1}^\infty u \pi(\Spn_s,\Diff u) \Big]
	\\&\leq 2 t_0 \Big( \sup_{x\in E} \int_{1}^\infty u \pi(x,\Diff u)\Big) \ProbSpnDec[\eta_T]
	\leq \frac{2 t_0}{\ProbSpnDec[\eta_{t_0}]} \Big( \sup_{x\in E} \int_{1}^\infty u \pi(x,\Diff u)\Big) \ProbSpnDec[\eta_T]^2.
\end{align}
	Note that from Step 1,
\[
	\frac{2 t_0}{\ProbSpnDec[\eta_{t_0}]} \Big( \sup_{x\in E} \int_{1}^\infty u \pi(x,\Diff u)\Big) < \infty.
\]

We now show that
	for any $T>t_0$, $\text{IV} \leq 4 \ProbSpnDec[\eta_T]^2$.
For fixed $T>t_0$, it follows from Fubini's theorem and \eqref{eq:K.41} that
\begin{align}
	&\text{IV}
	= \ProbSpnDec\Big[ 2\int_{-T}^{-t_0} \Diff t \int_{\frac{Ke^{-\epsilon t}}{\phi(\Spn_t)}}^\infty r \pi(\Spn_t, \Diff r) \int_{t+t_0}^0\Diff s\int_{\frac{Ke^{-\epsilon s}}{\phi(\Spn_s)}}^\infty u \pi(\Spn_s, \Diff u) \Big]
	\\&= 2\int_{-T}^{-t_0}\ProbSpnDec\Big[\int_{\frac{Ke^{-\epsilon t}}{\phi(\Spn_t)}}^\infty r \pi(\Spn_t, \Diff r) \int_{t+t_0}^0\Diff s\int_{\frac{Ke^{-\epsilon s}}{\phi(\Spn_s)}}^\infty u \pi(\Spn_s, \Diff u) \Big] \Diff t
	\\ & = 2\int_{-T}^{-t_0} \ProbSpnDec \bigg[ \ProbSpnDec\Big[ \int_{\frac{Ke^{-\epsilon t}}{\phi(\Spn_t)}}^\infty r \pi(\Spn_t, \Diff r) \int_{t+t_0}^0 \Diff s \int_{\frac{Ke^{-\epsilon s}}{\phi(\Spn_s)}}^\infty u \pi(\Spn_s, \Diff u) \Big| \tilde \srF_t \Big] \bigg]
	\Diff t
	\\ &  = 2\int_{-T}^{-t_0}\ProbSpnDec \bigg[ \int_{\frac{Ke^{-\epsilon t}}{\phi(\Spn_t)}}^\infty r \pi(\Spn_t, \Diff r) \int_{t+t_0}^0  \ProbSpnDec\Big[\int_{\frac{Ke^{-\epsilon s}}{\phi(\Spn_s)}}^\infty u \pi(\Spn_s, \Diff u) \Big| \tilde \srF_t \Big] \Diff s \bigg]\Diff t
	\\ &  = 2\int_{-T}^{-t_0}\ProbSpnDec \bigg[ \int_{\frac{Ke^{-\epsilon t}}{\phi(\Spn_t)}}^\infty r \pi(\Spn_t, \Diff r) \int_{t+t_0}^0  S_{s-t}h_s(\Spn_t) \Diff s \bigg]\Diff t
	\\\label{eq:E.938}&= 2\int_{-T}^{-t_0}\Diff t
\int_E\tilde \nu(\Diff y)
\int_{\frac{Ke^{-\epsilon t}}{\phi(y)}}^\infty r \pi(y, \Diff r) \int_{t+t_0}^0 S_{s-t}h_s (y) \Diff s,
\end{align}	
where, for any $t\in\RealLine$, $\tilde \srF_t := \sigma(\Spn_s:s\in (-\infty,t])$, and for any $s\leq 0$ and $y\in E$,
\begin{equation} \label{eq:E.939}
	h_s(y)
	:= \int_{\frac{Ke^{-\epsilon s}}{\phi(y)}}^\infty u \pi(y,\Diff u)
	\leq \sup_{x\in E}\int_{1}^\infty u \pi(x,\Diff u)
	<\infty.
\end{equation}
	Note for any $t\in (-T,-t_0)$ and $s\in (t+t_0,0)$, we have $s-t\geq t_0\geq t_2$,
	which, together with Step 3,
	implies that for any $y\in E$, $S_{s-t}h_s(y) \leq 2
 \tilde\nu(h_s)$.
	Therefore, from \eqref{eq:E.9375}, \eqref{eq:E.938} and \eqref{eq:E.939}, we have
\begin{align}
	&\text{IV}
	\leq 4 \int_{-T}^{-t_0}\Diff t
\int_E\tilde \nu(\Diff y)
\int_{\frac{Ke^{-\epsilon t}}{\phi(y)}}^\infty r \pi(y, \Diff r) \int_{t+t_0}^0
\tilde \nu(h_s) \Diff s
	\\ &\leq 4 \int_{-T}^0
\tilde \nu(h_t)
\Diff t \cdot \int_{-T}^0
\tilde \nu(h_s)\Diff s
	= 4\ProbSpnDec[\eta_T]^2.
\end{align}
	Now the desired result in this step follows.

\emph{Step 5.}
	We show that $\ProbSpnDec(\eta_{\infty}=\infty) >0$.
	Note that for any $T\geq 0$, from the Cauchy-Schwartz inequality,
\begin{align}
	&\sqrt{\ProbSpnDec[\eta_T^2] \ProbSpnDec\Big(\eta_T \geq \frac{1}{2}\ProbSpnDec[\eta_T]\Big)}
	\geq \ProbSpnDec[\eta_T\Ind_{\{\eta_T\geq \frac{1}{2}\ProbSpnDec[\eta_T]\}}]
      \\&= \ProbSpnDec[\eta_T] - \ProbSpnDec[\eta_T\Ind_{\{\eta_T< \frac{1}{2}\ProbSpnDec[\eta_T]\}} ]
      \geq \ProbSpnDec[\eta_T] - \ProbSpnDec\Big[\frac{1}{2}\ProbSpnDec[\eta_T]\Ind_{\{\eta_T< \frac{1}{2}\ProbSpnDec[\eta_T]\}}\Big]
	\geq \frac{1}{2}\ProbSpnDec[\eta_T].
\end{align}
	Let $t_0$ be as in Step 4.
	Then, there exists a $\C\label{const:E.99}>0$ such that for any $T\geq t_0$,
\begin{align}\label{eq:E.94}
	\ProbSpnDec\Big(\eta_\infty \geq \frac{1}{2}\ProbSpnDec[\eta_T]\Big)
	\geq \ProbSpnDec\Big(\eta_T \geq \frac{1}{2}\ProbSpnDec[\eta_T]\Big)
	\geq \frac{\ProbSpnDec[\eta_T]^2}{4\ProbSpnDec[\eta_T^2]}
	\geq \Cr{const:E.99}.
\end{align}
	By the monotone convergence theorem and Step 1, we have
\[
	\ProbSpnDec[\eta_T]
	\xrightarrow[T\to \infty]{} \ProbSpnDec[\eta_\infty]
	= \infty.
\]
Now by \eqref{eq:E.94} and the monotone convergence theorem again,
\begin{align}
	\ProbSpnDec(\eta_\infty =\infty)
	= \lim_{T\to \infty} \ProbSpnDec\Big(\eta_\infty \geq \frac{1}{2}\ProbSpnDec[\eta_T]\Big)
	\geq \Cr{const:E.99}.
\end{align}

\emph{Step 6.}
	We will show that
$\tilde \nu$
 is an ergodic measure of the semigroup $(S_t)_{t\geq 0}$ in the sense of \cite{MR1417491}*{Section 3.2}.
	To do this, we claim that for any
$\varphi \in L^2(\tilde \nu)$
satisfying $S_t \varphi = \varphi$ in
$L^2(\tilde \nu)$
 for all $t\geq 0$, it holds that $\varphi$ is a constant
$\tilde \nu$-a.e.
	In fact for
$\tilde \nu$-
almost every $x\in E$, from \eqref{asp:H2} and \eqref{eq:E.17}, we have
\begin{align}
	&\varphi(x)
	= S_t \varphi(x)
	= \frac{1}{e^{\lambda t}\phi(x)}T_t (\phi \varphi^+) (x) - \frac{1}{e^{\lambda t}\phi(x)}T_t (\phi \varphi^-) (x)
      \\&  =\nu(\phi \varphi^+) \big(1+ H_{t}(\phi \varphi^+)(x)\big) - \nu(\phi \varphi^-) \big(1+ H_{t}(\phi \varphi^-)(x)\big)
	\xrightarrow[t\to \infty]{} \nu( \phi \varphi),
\end{align}
	which implies the desired claim.
Now the desired result in this step follows from \cite{MR1417491}*{Theorem 3.2.4.}.

\emph{Step 7.}
	We will show that $\{\eta_\infty = \infty\}$ is an invariant event for this ergodic process $(\Spn_t)_{t\in \RealLine}$ under $\ProbSpnDec$ in the sense that, for any $t \in \RealLine$, $\ProbSpnDec(A_{0} \Delta A_{t}) = 0$ where
\[
	A_{t}
	:= \Big\{\int_{-\infty}^{0} \Diff s \int_{\frac{Ke^{-\epsilon s}}{\phi(\Spn_{s+t})}}^\infty y\pi(\Spn_{s+t},\Diff y) = \infty\Big\}, \quad t\in \RealLine.
\]
 We first claim that
\begin{equation}
 \label{eq:E.941}
	 A_r \subset A_{r-t}, \quad r\in \RealLine,t> 0.
\end{equation}
	In fact,  on the event $A_r$ we have
\begin{align}
	&\int_{-\infty}^0 \Diff s \int_{\frac{Ke^{-\epsilon s}}{\phi(\Spn_{s+r-t})}}^\infty y\pi(\Spn_{s+r-t},\Diff y)
	\geq \int_{-\infty}^0 \Diff s \int_{\frac{K e^{-\epsilon (s-t)}}{\phi(\Spn_{s+r-t})}}^\infty y\pi(\Spn_{s+r-t},\Diff y)
      \\&= \int_{-\infty}^{-t} \Diff s \int_{\frac{K e^{-\epsilon s}}{\phi(\Spn_{s+r})}}^\infty y\pi(\Spn_{s+r},\Diff y)
	\\&= \int_{-\infty}^0 \Diff s \int_{\frac{K e^{-\epsilon s}}{\phi(\Spn_{s+r})}}^\infty y\pi(\Spn_{s+r},\Diff y) - \int_{-t}^0 \Diff s \int_{\frac{K e^{-\epsilon s}}{\phi(\Spn_{s+r})}}^\infty y\pi(\Spn_{s+r},\Diff y)
	\\&\geq \int_{-\infty}^0 \Diff s \int_{\frac{K e^{-\epsilon s}}{\phi(\Spn_{s+r})}}^\infty y\pi(\Spn_{s+r},\Diff y) - |t| \cdot \sup_{x\in E} \int_1^\infty y\pi(x,\Diff y)
	= \infty,
\end{align}
as claimed.
We then claim that for
any
$r\in \RealLine$ and $t\geq 0$, $\ProbSpnDec$-almost surely,
\begin{align}
\label{eq:E.95}
	\int_{-\infty}^0 \Diff s \int_{\frac{Ke^{-\epsilon (s+t)}}{\phi(\Spn_{s+r+t})}}^\frac{Ke^{-\epsilon s}}{\phi(\Spn_{s+r+t})} y\pi(\Spn_{s+r+t},\Diff y)
	< \infty.
\end{align}
In fact, by \eqref{eq:K.41},
\begin{align}
	&\ProbSpnDec\Big[ \int_{-\infty}^0 \Diff s \int_{\frac{Ke^{-\epsilon (s+t)}}{\phi(\Spn_{s+r+t})}}^\frac{Ke^{-\epsilon s}}{\phi(\Spn_{s+r+t})} y\pi(\Spn_{s+r+t},\Diff y) \Big]
	 = \int_{-\infty}^0 \Diff s
\int_E \tilde \nu(\Diff x)
\int^{\frac{Ke^{-\epsilon s}}{\phi(x)}}_\frac{Ke^{-\epsilon (s+t)}}{\phi(x)} y\pi(x,\Diff y)
    \\&= \int_E \tilde \nu(\Diff x)
 \int^\infty_\frac{Ke^{-\epsilon t}}{\phi(x)} y\pi(x,\Diff y) \int_{-\frac{1}{\epsilon}
 \log \frac{y\phi(x) e^{\epsilon t}}{K}}^{-\frac{1}{\epsilon}\log\frac{y\phi(x)}{K}} \Diff s
	\\&= t\int_E
\tilde \nu(\Diff x)
\int^\infty_\frac{Ke^{-\epsilon t}}{\phi(x)} y\pi(x,\Diff y)
	\leq t \cdot \sup_{x\in E} \int^\infty_{e^{-\epsilon t}} y\pi(x,\Diff y)
	< \infty
\end{align}
which implies the claim.
Finally, we claim that
	\begin{equation} \label{eq:E.951}
	A_r \cap \Omega_{r,t} \subset A_{r+t}, \quad r\in \RealLine, t> 0,
\end{equation}
where $\Omega_{r,t}$ is the event that \eqref{eq:E.95} holds.
	In fact,  on the event $A_r\cap \Omega_{r,t}$ we have
\begin{align}
	&\int_{-\infty}^0 \Diff s \int_{\frac{K e^{-\epsilon s}}{\phi(\Spn_{s+r+t})}}^\infty y\pi(\Spn_{s+r+t},\Diff y)
	\\&= \int_{-\infty}^0 \Diff s \int_{\frac{K e^{-\epsilon (s+t)}}{\phi(\Spn_{s+r+t})}}^\infty y\pi(\Spn_{s+r+t},\Diff y) - \int_{-\infty}^0 \Diff s \int_{\frac{Ke^{-\epsilon (s+t)}}{\phi(\Spn_{s+r+t})}}^\frac{Ke^{-\epsilon s}}{\phi(\Spn_{s+r+t})} y\pi(\Spn_{s+r+t},\Diff y)
	\\&= \int_{-\infty}^t \Diff s \int_{\frac{Ke^{-\epsilon s}}{\phi(\Spn_{s+r})}}^\infty y\pi(\Spn_{s+r},\Diff y) -  \int_{-\infty}^0 \Diff s \int_{\frac{Ke^{-\epsilon (s+t)}}{\phi(\Spn_{s+r+t})}}^\frac{Ke^{-\epsilon s}}{\phi(\Spn_{s+r+t})} y\pi(\Spn_{s+r+t},\Diff y)
	\\&\geq \int_{-\infty}^0 \Diff s \int_{\frac{Ke^{-\epsilon s}}{\phi(\Spn_{s+r})}}^\infty y\pi(\Spn_{s+r},\Diff y) -  \int_{-\infty}^0 \Diff s \int_{\frac{Ke^{-\epsilon (s+t)}}{\phi(\Spn_{s+r+t})}}^\frac{Ke^{-\epsilon s}}{\phi(\Spn_{s+r+t})} y\pi(\Spn_{s+r+t},\Diff y)
	= \infty
\end{align}
	as claimed.
	Now, for any $r<t$ in $\RealLine$,
	from \eqref{eq:E.941},
	we know that $\ProbSpnDec(A_t\setminus A_r) = 0$;
	from \eqref{eq:E.95} and \eqref{eq:E.951},
	we know that $\ProbSpnDec(A_r\setminus A_t) = 0$.
	Therefore, the desired result in this step follows.

\emph{Final Step.}
From steps 5, 7
and \cite{MR1417491}*{Theorem 1.2.4.(i)}, we get that
\begin{equation} \label{eq:E.96}
	\int_{-\infty}^{0} \Diff s \int_{\frac{K e^{-\epsilon s}}{\phi(\Spn_{s})}}^\infty y \pi(\Spn_{s},\Diff y) = \infty, \quad \ProbSpnDec\text{-a.s.}
\end{equation}
	From \eqref{eq:K.41}, we know that $(\Spn_s)_{s\in \RealLine}$ has the same distribution as $(\Spn_{s-s_0})_{s\in \RealLine}$.
	Therefore we have from \eqref{eq:E.96} that
\begin{equation}
	\int_{-\infty}^{0} \Diff s \int_{\frac{K e^{-\epsilon s}}{\phi(\Spn_{s-s_0})}}^\infty y \pi(\Spn_{s-s_0},\Diff y) = \infty, \quad \ProbSpnDec\text{-a.s.}
\end{equation}
	Now we have
$\ProbSpnDec\text{-a.s.}$,
\begin{align}
	&\int_{-\infty}^{-s_0} \Diff s \int_{\frac{e^{-\epsilon s}}{\phi(\Spn_{s})}}^\infty y \pi(\Spn_{s},\Diff y)
	= \int_{-\infty}^{0} \Diff s \int_{\frac{e^{-\epsilon s} e^{\epsilon s_0}}{\phi(\Spn_{s-s_0})}}^\infty y \pi(\Spn_{s-s_0},\Diff y)
	\\&\geq \int_{-\infty}^{0} \Diff s \int_{\frac{K e^{-\epsilon s}}{\phi(\Spn_{s-s_0})}}^\infty y \pi(\Spn_{s-s_0},\Diff y)
	= \infty
\end{align}
as desired.
\end{proof}

\subsection*{Acknowledgment}
Part of the research for this paper was done while the third-named author was visiting
Jiangsu Normal University, where he was partially supported by  a grant from
the National Natural Science Foundation of China (11931004) and by
the Priority Academic Program Development of Jiangsu Higher Education Institutions.


\begin{bibdiv}
	\begin{biblist}
		
		\bib{MR0373040}{book}{
			author={Athreya, K. B.},
			author={Ney, P. E.},
			title={Branching processes},
			note={Die Grundlehren der mathematischen Wissenschaften, Band 196},
			publisher={Springer-Verlag, New York-Heidelberg},
			date={1972},
			pages={xi+287},
			review={\MR{0373040}},
		}
		
		\bib{MR2399296}{article}{
			author={Champagnat, N.},
			author={R\oe lly, S.},
			title={Limit theorems for conditioned multitype Dawson-Watanabe processes
				and Feller diffusions},
			journal={Electron. J. Probab.},
			volume={13},
			date={2008},
			pages={no. 25, 777--810},
			review={\MR{2399296}},
		}
		
		\bib{MR2397881}{article}{
			author={Chen, Z.-Q.},
			author={Ren, Y.-X.},
			author={Wang, H.},
			title={An almost sure scaling limit theorem for Dawson-Watanabe superprocesses},
			journal={J. Funct. Anal.},
			volume={254},
			date={2008},
			number={7}, pages={1988--2019}, issn={0022-1236},
			review={\MR{2397881}},
		}

     \bib{MR4264652}{article}{
            author={Cox, A. M. G.},
			author={Horton, E.},
			author={Kyprianou, A. E.},
            author={Villemonais, D.},
			title={Stochastic methods for neutron transport equation III: Generational many-to-one and $\kappa_{eff}$},
			journal={SIAM J. Appl. Math.},
			date={2021},
            volume={81}, number={3}, pages={982--1001},
            review={\MR{4264652}},
	}
		
		\bib{MR1417491}{book}{
			author={Da Prato, G.},
			author={Zabczyk, J.},
			title={Ergodicity for infinite-dimensional systems},
			series={London Mathematical Society Lecture Note Series},
			volume={229},
			publisher={Cambridge University Press, Cambridge},
			date={1996},
			pages={xii+339},
			isbn={0-521-57900-7},
			review={\MR{1417491}},
		}
		\bib{MR3035765}{article}{
			author={Delmas, J.},
			author={H\'{e}nard, O.},
			title={A Williams decomposition for spatially dependent super-processes},
			journal={Electron. J. Probab.},
			volume={18},
			date={2013},,
			pages={no. 37, 43 pp},
			review={\MR{3035765}},
		}
		\bib{MR3395469}{article}{
			author={Eckhoff, M.},
			author={Kyprianou, A. E.},
			author={Winkel, M.},
			title={Spines, skeletons and the strong law of large numbers for
				superdiffusions},
			journal={Ann. Probab.},
			volume={43},
			date={2015},
			number={5},
			pages={2545--2610},
			issn={0091-1798},
			review={\MR{3395469}},
		}
		\bib{MR2040776}{article}{
			author={Engl\"{a}nder, J.},
			author={Kyprianou, A. E.},
			title={Local extinction versus local exponential growth for spatial
				branching processes},
			journal={Ann. Probab.},
			volume={32},
			date={2004},
			number={1A},
			pages={78--99},
			issn={0091-1798},
			review={\MR{2040776}},
		}
		\bib{MR2006204}{article}{
			author={Etheridge, A. M.},
			author={Williams, D. R. E.},
			title={A decomposition of the $(1+\beta)$-superprocess conditioned on
				survival},
			journal={Proc. Roy. Soc. Edinburgh Sect. A},
			volume={133},
			date={2003},
			number={4},
			pages={829--847},
			issn={0308-2105},
			review={\MR{2006204}},
		}
			\bib{MR1249698}{article}{
		author={Evans, S. N.},
		title={Two representations of a conditioned superprocess},
		journal={Proc. Roy. Soc. Edinburgh Sect. A},
		volume={123},
		date={1993},
		number={5},
		pages={959--971},
		issn={0308-2105},
		review={\MR{1249698}},
	}
		\bib{MR1088825}{article}{
			author={Evans, S. N.},
			author={Perkins, E.},
			title={Measure-valued Markov branching processes conditioned on
				nonextinction},
			journal={Israel J. Math.},
			volume={71},
			date={1990},
			number={3},
			pages={329--337},
			issn={0021-2172},
			review={\MR{1088825}},
		}
	\bib{garcia2022asymptotic}{article}{
		title={Asymptotic moments of spatial branching processes},
		author={Garcia, I. G.},
		author={Horton, E.},
		author={Kyprianou, A. E.},
		journal={Probability Theory and Related Fields},
		date={2022},
		publisher={Springer New York},
	}
		\bib{MR408016}{article}{
			author={Grey, D. R.},
			title={Asymptotic behaviour of continuous time, continuous state-space
				branching processes},
			journal={J. Appl. Probability},
			volume={11},
			date={1974},
			pages={669--677},
			issn={0021-9002},
			review={\MR{408016}},
		}


\bib{MR4187129}{article}{
            author={Harris, S. C..},
			author={Horton, E.},
			author={ Kyprianou, A. E.},
			title={Stochastic methods for the neutron transport equation
II: almost sure growth},
			journal={Ann. Appl. Probab.},
			volume={30},
			date={2020},
			pages={2815--2845},
		review={\MR{4187129}},
}	



	\bib{harris2021yaglom}{article}{
		title={Yaglom limit for critical neutron transport},
		author={Harris, S. C.},
		author={Horton, E.},
		author={Kyprianou, A. E.},
		author={Wang, M.},
journal={To appear in Ann. Probab.},
		eprint={https://doi.org/10.48550/arXiv.2103.02237 },
	}
		\bib{MR0217889}{article}{
			author={Heathcote, C. R.},
			author={Seneta, E.},
			author={Vere-Jones, D.},
			title={A refinement of two theorems in the theory of branching processes},
			language={English, with Russian summary},
			journal={Teor. Verojatnost. i Primenen.},
			volume={12},
			date={1967},
			pages={341--346},
			issn={0040-361x},
			review={\MR{0217889}},
		}
			\bib{MR1974383}{book}{
		author={Hern\'{a}ndez-Lerma, O.},
		author={Lasserre, J. B.},
		title={Markov chains and invariant probabilities},
		series={Progress in Mathematics},
		volume={211},
		publisher={Birkh\"{a}user Verlag, Basel},
		date={2003},
		pages={xvi+205},
		isbn={3-7643-7000-9},
		review={\MR{1974383}},
	}

\bib{MR4187122}{article}{
			author={Horton, E.},
			author={ Kyprianou, A. E.},
			author={Villemonais, D.},
			title={Stochastic methods for the neutron transport equation
I: linear semigroup asymptotics},
			journal={Ann. Appl. Probab.},
			volume={30},
			date={2020},
			pages={2573--2612},
		review={\MR{4187122}},
}	

		\bib{MR205337}{article}{
			author={Joffe, A.},
			title={On the Galton-Watson branching process with mean less than one},
			journal={Ann. Math. Statist.},
			volume={38},
			date={1967},
			pages={264--266},
			issn={0003-4851},
			review={\MR{205337}},
		}
\bib{MR4226142}{book}{
	author={Kallenberg, O.},
	title={Foundations of modern probability},
	series={Probability Theory and Stochastic Modelling},
	volume={99},
	note={Third edition [of  1464694]},
	publisher={Springer, Cham},
	date={[2021] \copyright 2021},
	pages={xii+946},
	isbn={978-3-030-61871-1},
	isbn={978-3-030-61870-4},
	review={\MR{4226142}},
}
	\bib{MR2487824}{article}{
		author={Kim, P.},
		author={Song, R.},
		title={Intrinsic ultracontractivity of non-symmetric diffusion semigroups
			in bounded domains},
		journal={Tohoku Math. J. (2)},
		volume={60},
		date={2008},
		number={4},
		pages={527--547},
		issn={0040-8735},
		review={\MR{2487824}},
	}
		\bib{MR3155252}{book}{
			author={Kyprianou, A. E.},
			title={Fluctuations of L\'{e}vy processes with applications},
			series={Universitext},
			edition={2},
			note={Introductory lectures},
			publisher={Springer, Heidelberg},
			date={2014},
			pages={xviii+455},
			isbn={978-3-642-37631-3},
			isbn={978-3-642-37632-0},
			review={\MR{3155252}},
		}
		\bib{MR2299923}{article}{
			author={Lambert, A.},
			title={Quasi-stationary distributions and the continuous-state branching
				process conditioned to be never extinct},
			journal={Electron. J. Probab.},
			volume={12},
			date={2007},
			pages={no. 14, 420--446},
			issn={1083-6489},
			review={\MR{2299923}},
		}
		\bib{MR0228073}{article}{
			author={Lamperti, J.},
			author={Ney, P.},
			title={Conditioned branching processes and their limiting diffusions},
			language={English, with Russian summary},
			journal={Teor. Verojatnost. i Primenen.},
			volume={13},
			date={1968},
			pages={126--137},
			issn={0040-361x},
			review={\MR{0228073}},
		}
		\bib{MR1727226}{article}{
			author={Li, Z.},
			title={Asymptotic behaviour of continuous time and state branching
				processes},
			journal={J. Austral. Math. Soc. Ser. A},
			volume={68},
			date={2000},
			number={1},
			pages={68--84},
			issn={0263-6115},
			review={\MR{1727226}},
		}
		\bib{MR2760602}{book}{
			author={Li, Z.},
			title={Measure-valued branching Markov processes},
			series={Probability and its Applications (New York)},
			publisher={Springer, Heidelberg},
			date={2011},
			pages={xii+350},
			isbn={978-3-642-15003-6},
			review={\MR{2760602}},
		}
		\bib{MR2535827}{article}{
			author={Liu, R.-L.},
			author={Ren, Y.-X.},
			author={Song, R.},
			title={$L\log L$ criterion for a class of superdiffusions},
			journal={J. Appl. Probab.},
			volume={46},
			date={2009},
			number={2},
			pages={479--496},
			issn={0021-9002},
			review={\MR{2535827}},
		}
		
		\bib{MR3010225}{article}{
			author={Liu, R.-L.},
			author={Ren, Y.-X.},
			author={Song, R.},
			title={Strong law of large numbers for a class of superdiffusions},
			journal={Acta Appl. Math.},
			volume={123},
			date={2013},
			pages={73--97}, issn={0167-8019},
			review={\MR{3010225}},
		}
		
		\bib{MR4175472}{article}{
			author={Liu, R.},
			author={Ren, Y.-X.},
			author={Song, R.},
			author={Sun, Z.},
			title={Quasi-stationary distributions for subcritical superprocesses},
			journal={Stochastic Process. Appl.},
			volume={132},
			date={2021},
			pages={108--134},
			issn={0304-4149},
			review={\MR{4175472}},
		}
		\bib{MR1349164}{article}{
			author={Lyons, R.},
			author={Pemantle, R.},
			author={Peres, Y.},
			title={Conceptual proofs of $L\log L$ criteria for mean behavior of
				branching processes},
			journal={Ann. Probab.},
			volume={23},
			date={1995},
			number={3},
			pages={1125--1138},
			issn={0091-1798},
			review={\MR{1349164}},
		}
		\bib{MR2994898}{article}{
			author={M\'{e}l\'{e}ard, S.},
			author={Villemonais, D.},
			title={Quasi-stationary distributions and population processes},
			journal={Probab. Surv.},
			volume={9},
			date={2012},
			pages={340--410},
			review={\MR{2994898}},
		}
		\bib{MR1735780}{article}{
			author={Pakes, A. G.},
			title={Revisiting conditional limit theorems for the mortal simple
				branching process},
			journal={Bernoulli},
			volume={5},
			date={1999},
			number={6},
			pages={969--998},
			issn={1350-7265},
			review={\MR{1735780}},
		}
		\bib{penisson2010conditional}{thesis}{
		title={Conditional limit theorems for multitype branching processes and illustration in epidemiological risk analysis},
		author={P\'{e}nisson, S.},
		date={2010},
		type={Ph.D. Thesis},
		organization={Universit{\"a}t Potsdam; Universit{\'e} Paris Sud-Paris XI},
		eprint={https://tel.archives-ouvertes.fr/tel-00570458/document },
	}
		\bib{MR3476213}{article}{
			author={P\'{e}nisson, S.},
			title={Beyond the $Q$-process: various ways of conditioning the multitype
				Galton-Watson process},
			journal={ALEA Lat. Am. J. Probab. Math. Stat.},
			volume={13},
			date={2016},
			number={1},
			pages={223--237},
			review={\MR{3476213}},
		}
		\bib{pollett2008quasi}{article}{
			author={Pollett, P. K.},
			title={Quasi-stationary distributions: a bibliography},
			date={2015},
			eprint={http://www.maths.uq.edu.au/*pkp/papers/qsds/qsds.html},
		}
		\bib{MR4058118}{article}{
			author={Ren, Y.-X.},
			author={Song, R.},
			author={Sun, Z.},
			title={Spine decompositions and limit theorems for a class of critical
				superprocesses},
			journal={Acta Appl. Math.},
			volume={165},
			date={2020},
			pages={91--131},
			issn={0167-8019},
			review={\MR{4058118}},
		}
		\bib{MR4102269}{article}{
			author={Ren, Y.-X.},
			author={Song, R.},
			author={Sun, Z.},
			title={Limit theorems for a class of critical superprocesses with stable
				branching},
			journal={Stochastic Process. Appl.},
			volume={130},
			date={2020},
			number={7},
			pages={4358--4391},
			issn={0304-4149},
			review={\MR{4102269}},
		}
		\bib{MR4049077}{article}{
			author={Ren, Y.-X.},
			author={Song, R.},
			author={Sun, Z.},
			author={Zhao, J.},
			title={Stable central limit theorems for super Ornstein-Uhlenbeck
				processes},
			journal={Electron. J. Probab.},
			volume={24},
			date={2019},
			pages={Paper No. 141, 42},
			review={\MR{4049077}},
		}
		\bib{MR4397885}{article}{
		author={Ren, Y.-X.},
		author={Song, R.},
		author={Sun, Z.},
		author={Zhao, J.},
		title={Stable central limit theorems for super Ornstein-Uhlenbeck
			processes, II},
		journal={Acta Math. Sin. (Engl. Ser.)},
		volume={38},
		date={2022},
		number={3},
		pages={487--498},
		issn={1439-8516},
		review={\MR{4397885}},
	}
	\bib{MR4359790}{article}{
		author={Ren, Y.-X.},
		author={Song, R.},
		author={Yang, T.},
		title={Spine decomposition and $L \log L$ criterion for superprocesses
			with non-local branching mechanisms},
		journal={ALEA Lat. Am. J. Probab. Math. Stat.},
		volume={19},
		date={2022},
		number={1},
		pages={163--208},
		review={\MR{4359790}},
	}
		\bib{MR3293289}{article}{
			author={Ren, Y.-X.},
			author={Song, R.},
			author={Zhang, R.},
			title={Central limit theorems for supercritical superprocesses},
			journal={Stochastic Process. Appl.},
			volume={125},
			date={2015},
			number={2},
			pages={428--457}, issn={0304-4149},
			review={\MR{3293289}},
		}
		\bib{MR3459635}{article}{
			author={Ren, Y.-X.},
			author={Song, R.},
			author={Zhang, R.},
			title={Limit theorems for some critical superprocesses},
			journal={Illinois J. Math.},
			volume={59},
			date={2015},
			number={1},
			pages={235--276},
			issn={0019-2082},
			review={\MR{3459635}},
		}
		\bib{MR3601657}{article}{
		author={Ren, Y.-X.},
		author={Song, R.},
		author={Zhang, R.},
		title={Central limit theorems for supercritical branching nonsymmetric
			Markov processes},
		journal={Ann. Probab.},
		volume={45},
		date={2017},
		number={1},
		pages={564--623},
		issn={0091-1798},
		review={\MR{3601657}},
	}
		\bib{MR958914}{book}{
			author={Sharpe, M.},
			title={General theory of Markov processes},
			series={Pure and Applied Mathematics},
			volume={133},
			publisher={Academic Press, Inc., Boston, MA},
			date={1988},
			pages={xii+419},
			isbn={0-12-639060-6},
			review={\MR{958914}},
		}
		\bib{MR0022045}{article}{
			author={Yaglom, A. M.},
			title={Certain limit theorems of the theory of branching random
				processes},
			language={Russian},
			journal={Doklady Akad. Nauk SSSR (N.S.)},
			volume={56},
			date={1947},
			pages={795--798},
			review={\MR{0022045}},
		}
	\end{biblist}
\end{bibdiv}
\end{document}